\documentclass[11pt,a4paper]{article}
\usepackage[utf8]{inputenc}
\usepackage{amsmath}
\usepackage{amsfonts}
\usepackage{amssymb}
\usepackage{amsthm}
\usepackage{stmaryrd}
\makeatletter
\newcommand{\norm}[1]{\left\lVert#1\right\rVert}
\newtheorem{thm}{Theorem}[section]
\newtheorem{hyp}[thm]{Hypothesis}

\newtheorem{rem}[thm]{Remark}
\newtheorem{lem}[thm]{Lemma}
\newtheorem{prop}[thm]{Proposition}
\newtheorem{cor}[thm]{Corollary}
\newcommand{\divergence}{\mathop{\rm div}\nolimits}

\@addtoreset{equation}{section}
\makeatother
\begin{document}

\title{Transport Maps for $\beta$-Matrix Models in the Multi-Cut Regime}
\date{}
\author{Florent Bekerman \thanks{\noindent Department  of Mathematics, Massachusetts Institute of Technology. Email: bekerman@mit.edu}}
\maketitle
\begin{abstract}
\noindent We use the transport  methods developped in \cite{BFG} to obtain universality results for local statistics of eigenvalues in the bulk and at the edge for $\beta$-matrix models in the multi-cut regime.  We  construct an approximate transport map inbetween two probability measures from the fixed filling fraction model discussed in \cite{BGK} and deduce from it universality in  the initial model.
\end{abstract}

\section{Introduction}
\noindent The goal of this paper is to obtain  universality results for  local statistics of the eigenvalues for $\beta$-matrix models. The analysis of the local fluctuations of the eigenvalues was  first done for the GUE and after the pioneer  work of Gaudin, Dyson and Mehta  the sine kernel law was exhibited  (see \cite{Meh}). Universality was then  shown for  classical values of  $ \beta $  ($ \beta \in \{ 1,2,4 \}$)  and smooth potentials through the study of orthogonal polynomials (See the work of L. Pastur and M. Shcherbina \cite{PSI} \cite{PSII}, and P. Deift et al. \cite{DGa},\cite{DGb} ). \\
\noindent  For non classical values of $ \beta $ and unless the potential is quadratic, there is however no known matrix representation behind the model and universality results cannot be obtained through orthogonal polynomial methods.  For a quadratic potential, the log-gases can be viewed as the eigenvalues of tridiagonal matrices (see \cite{DE}) and the local behaviour of the eigenvalues in the bulk and at the edge have been made explicit thanks to the work of  B. Vir\'{a}g, B. Valk\'{o}, J. Ramirez and B. Rider \cite{VV}, \cite{RRV}.\\
\noindent Recently, new techniques have been developed to study universality of the fluctuations. Thus, P. Bourgade, L. Erdös and H.T. Yau use dynamical methods and Dirichlet form estimates in \cite{BEYI},\cite{BEYIII}  to obtain  the averaged energy universality of  the correlations functions and fixed  gap universality in the bulk (for $\beta > 0 $), as well as universality at the edge ($\beta \geq  1 $, see \cite{BEYII}) for smooth one-cut potentials. In the paper \cite{ShII}, M. Shcherbina uses change of variables to obtain  the averaged energy universality of  the correlation functions in both the one-cut case and multi-cut cases. The fluctuations of the linear statistics of the eigenvalues in the multi-cut regime where  studied in \cite{BGII} and \cite{ShI}, and rigidity in the multi-cut regime was recently obtained in \cite{YLI}. In the paper \cite{BFG},    A.Figalli,  A.Guionnet and the author construct  approximate transport maps with an accurate dependence in the dimension.  The dependence in $N$  allows  to compare the local fluctuation of the eigenvalues under two different potentials. The potentials do not need to be analytic, but an important  hypothesis made in this previous article was the connectedness of the support of the limit  of the spectral measure . Here, we  assume that the potentials are analytic but remove the one-cut assumption and use the same methods to construct approximate transport maps in the case where the filling fractions of each cut is fixed. As a result, we  obtain  universality of  fixed eigenvalue gaps at the edge and in the bulk . The plan of this paper is as follows: In the first section we intoduce some notations and state our main results.  We   reintroduce in section 2 a more general model discussed in \cite{BGK} of $\beta$ log-gases with Coulomb interaction and  construct  an approximate transport map between two measures from this model when the number of particles in each cut is fixed. We will see how  this approximate transport  can lead to  universality results   in the fixed filling fractions case, and conclude for the initial model in Section 4. The main results are Theorems \ref{theoreme},  \ref{theoremedge} and      \ref{theoremedge2}. \\

\noindent  We consider the general $\beta$-matrix model.
\noindent For a subset $\mathcal{A}$ of $\mathbb{R}$ union of disjoint (possibly semi-infinite or infinite) intervals and  a  potential $V : \mathcal{A} \longrightarrow \mathbb{R}$  and $\beta > 0$, we denote the measure on $\mathcal{A}^N$
\begin{equation}
\mathbb{P}_{V,\mathcal{A}}^N \ (d\lambda_1,\cdots,d\lambda_N):= \frac{1}{Z_{V,\mathcal{A}}^N}\prod_{1\leq i < j \leq N} \lvert \lambda_i - \lambda_j \rvert^\beta \exp \Big( {- N \sum_{1 \leq i \leq N} V(\lambda_i)} \Big) \prod d\lambda_i \ ,
\end{equation}

\noindent with  \[{Z_{V,\mathcal{A}}^N}  = \int_{\mathcal{A}^N} \prod_{1\leq i < j \leq N} \lvert \lambda_i - \lambda_j \rvert^\beta \exp \Big( {- \sum_{1 \leq i  \leq N} V(\lambda_i)} \Big) \prod d\lambda_i  . \]  

\noindent It is well known (see  \cite{AGZ}, \cite{AG} and \cite{DeiI}) that under $\mathbb{P}_{V,\mathcal{A}}^N$ the empirical measure of the eigenvalues converge towards an equilibrium measure:

\begin{prop}
Assume that $V : \mathcal{A}  \longrightarrow \mathbb{R}$ is continuous and if $\infty \in \mathcal{A}$ assume that 

\begin{equation*}
\liminf_{x \rightarrow \infty} \frac{V(x)}{\beta \log |x|} >1 .
\end{equation*} \\

\noindent then the energy defined by 
\begin{equation}\label{energy}
E(\mu) = \int V(x) d\mu(x) - \frac{\beta}{2}  \log |x_1-x_2 | d\mu(x_1) d\mu(x_2)
\end{equation}

\noindent has a unique global minimum on the space $\mathcal{M}_1(\mathcal{A})$ of probability measures on $\mathcal{A}$.\\

\noindent Moreover, under $\mathbb{P}_{V,\mathcal{A}}^N$ the normalized empirical measure $L_N = N^{-1} \sum_{i=1}^N \delta_{\lambda_i}$ converges almost surely and in expectation towards the unique probability measure $\mu_{V} $ which minimizes the energy.

\noindent It has compact support $A$ and it is uniquely determined by the existence of a constant $C$ such that:

\begin{equation*}
\beta \int_\mathcal{A}\log |x-y|  d\mu_{V}(y) - V(x) \leq C  \ ,
\end{equation*}

\noindent with equality almost everywhere on the support. The support of  $\mu_{V}$  is a union of intervals  $A = \underset {{0\leq h \leq g}}{\bigcup} [\alpha_{h,-} ; \alpha_{h,+}]$  with $\alpha_{h,-} < \alpha_{h,+}$  and if $V$ is analytic on a neighbourhood of $A$,
\begin{equation*}
\frac{d\mu_{V}}{dx}= S(x)\prod_{h=0}^g \sqrt{\lvert x - \alpha_{h,-} \rvert \lvert x - \alpha_{h,+} \rvert} \ , 
\end{equation*}
\noindent with $S$  analytic on a neighbourhood of $A$.

\end{prop}

\noindent We make the following assumptions:
\begin{hyp}\label{hypo}
\end{hyp}
\begin{itemize}
\item  $V$  is  continuous and goes to infinity faster than  $\beta \ \log \rvert x \lvert$ if $\mathcal{A}$ is semi-infinite.
\item The support of  $\mu_{V}$ is a union of $g+1$ intervals  $A = \underset {{0\leq h \leq g}}{\bigcup} A_h$ with $A_h= [\alpha_{h,-} ; \alpha_{h,+}]$,  $\alpha_{h,-} < \alpha_{h,+}$  and
\begin{equation}
\frac{d\mu_{V}}{dx}= \rho_V(x)= S(x)\prod_{h=0}^g \sqrt{\lvert x - \alpha_{h,-} \rvert \lvert x - \alpha_{h,+} \rvert}  \ \ \ \ with \  S > 0 \ on \ [\alpha_{h,-} ; \alpha_{h,+}] .
\end{equation}

\item $V$  extends to an holomorphic function on a open  neighborhood  $U$ of $A$, $U = \underset {{0\leq h \leq g}}{\bigcup} U_h$ and $A_h \subset U_h$ 

\item The function $ V(\cdot) - \beta \int_\mathcal{A}\log |\cdot-y|  d\mu_{V}(y) $ achieves its minimum on the support only.

\end{itemize}

\noindent The last hypothesis is useful to ensure a control of large deviations. Before stating the main theorems, we will introduce some  notations.\\

\noindent \textbf{Notations}
\begin{itemize}
\item For all $ 0 \leq h \leq g$, $ \epsilon_{\star,h} = \mu_{V} (A_h)$ and $\boldsymbol{\epsilon_\star} = (\epsilon_{\star, 0}, \cdots,\epsilon_{\star, g})  $.
\item For all $ 0 \leq h \leq g$, $ N_{\star,h} = N \epsilon_{\star,h}$, $\boldsymbol{N_\star} = N \boldsymbol{\epsilon_\star}$, and $ \lfloor \boldsymbol{N_\star} \rfloor = (\lfloor N \epsilon_{\star,0}  \rfloor  , \cdots, \lfloor N \epsilon_{\star,g}  \rfloor)  $.

\item For a configuration $\boldsymbol{\lambda} \in \mathbb{R}^N$,   $N(\boldsymbol{\lambda})$ denotes the vector such that for all $0 \leq h \leq g$,  $ (N(\boldsymbol{\lambda}))_h$ is the number of eigenvalues in $U_h$.
\item For an index $i$, we introduce the classical location $E^{V,N}_{i}$ of the $i-th$ eigenvalue by 
\begin{equation*}
\int_{- \infty}^{E^{V,N}_{i}} \rho_V(x) dx = \frac{i}{N}  .
\end{equation*}

\noindent In the case where the fraction $i/N$ exactly equals to the sum of the mass of the first cuts, we consider the smallest  $E$ satisfying the equality.

\item For a configuration $\boldsymbol{\lambda} \in \mathbb{R}^N$,   let $\lambda_{h,i}$ the i-th smallest eigenvalue in $U_h$.
\item For a vector $\mathbf{x} \in \mathbb{R}^{g+1}$ and $ 0 \leq h \leq g$, $[\mathbf{x}]_h = x_0 + \cdots + x_h$ and $[\mathbf{x}]_{-1}=0$ .
\item For a vector $\mathbf{x} \in \mathbb{R}^{g+1}$,    $ 0 \leq h \leq g$ and $i \in \mathbb{N}$ we write  $i[h,\mathbf{x}] = i - [\mathbf{x}]_{h-1}$.
\item For a   signed measure $\nu$ and a  function $f \in L^1(d|\nu|)$  we will write $\nu(f) = \int f d\nu$.
\end{itemize}

\noindent The main goal of this paper is to prove universality results in the bulk and at the edge. 

\noindent Fixed eigenvalue gaps have been proved to be universal  for regular one-cut potentials (see \cite{BFG}, \cite{BEYI}), and their convergence can be obtained using the translation invariance of the eigenvalue gaps as in \cite{EY} (see also \cite{Tao} for the case of the GUE).  More precisely,  if $V$ is the Gaussian potential $G(\lambda):=\beta \frac{\lambda^2}{4}$ we have for $i$ away from the edge

\begin{equation}\label{convergencegap}
N  \rho_V(E^{V,N}_{i}) (\lambda_{i+1} - \lambda_i) \xrightarrow{\ \ \ \mathcal{L} \ \ \ } \mathcal{G_\beta} \ ,
\end{equation}

\noindent where  $\mathcal{G_\beta}$  is some distribution (corresponding to the Gaudin distribution  for $\beta=2$). \\

\noindent  Our first Theorem states   that this result  holds for any multi-cut potential satisfying  Hypothesis \ref{hypo}.

\begin{thm}\label{theoreme}
Let $\beta >0 $ and assume that $V$ satisfies Hypothesis \ref{hypo}.\\

\noindent Let $i \leq N$ such that for some  $\varepsilon > 0$  and  $h\in \llbracket 0 ; g  \rrbracket$ $,  \varepsilon N < i - [\boldsymbol{N_\star}]_{h-1} < N_{\star,h} - \varepsilon N $. Then

\begin{equation*}
N  \rho_V(E^{V,N}_{i}) (\lambda_{i+1} - \lambda_i) \xrightarrow{\ \ \ \mathcal{L} \ \ \ } \mathcal{G_\beta} .
\end{equation*}
\end{thm}

\noindent We now state the results at  the edge. Under a Gaussian potential and for general $\beta$, the behaviour of the eigenvalues at  the edge is described by the Stochastic Airy Operator (We refer to \cite{RRV}). J. Ram\'{i}rez, B. Rider and B. Vir\'{a}g have shown that under the Gaussian potential, the $k$ first rescaled  eigenvalues $(N^{2/3}(\lambda_1 +2), \cdots, N^{2/3}(\lambda_k +2))$ converge in distribution to $(\Lambda_1, \cdots, \Lambda_k)$ where $\Lambda_i$ is the $i$-th smallest eigenvalue of the stochastic Airy operator $SAO_\beta$.\\ 

\noindent In the following result,   $\Phi^{h}$  are smooth transport maps (defined later).

\begin{thm}\label{theoremedge}
Assume that $V$ satisfies Hypothesis \ref{hypo}.  Let $\tilde{\mathbb{P}}_{V,\mathcal{A}}^{N}$ denote the distribution of the ordered eigenvalues under  $ \mathbb{P}_{V,\mathcal{A}}^{N}$.\\ 

\noindent If  for all $0 \leq h \leq g$ $f_h: \mathbb{R}^m\longrightarrow \mathbb{R}$  is Lipschitz and compactly supported we have:

\begin{equation*}
\begin{split}
& \lim_{N\rightarrow \infty}  \int \prod_{0 \leq h \leq g}  f_h \big(N^{2/3}(\lambda_{h,1}-  \alpha_{h,-}),  \cdots,N^{2/3}(\lambda_{h,m} - \alpha_{h,-})\big)d\tilde{\mathbb{P}}_{V,\mathcal{A}}^{N} \\
   &= \prod_{0 \leq h \leq g}  \mathbb{E}_{SAO_\beta}\  f_h (\Phi^{h}(-2) \Lambda_1 , \cdots , \Phi^{h}(-2) \Lambda_m) .
\end{split}
\end{equation*}

\end{thm}

\noindent It is also interesting to study the behaviour of the $i-th$ eigenvalue where $i= [\lfloor \boldsymbol{N_\star} \rfloor]_{h-1}+1$. This eigenvalue would be typically  located at the right edge of the $h$-th cut or the left edge of the $h+1$-st cut. The following theorem gives the limiting distribution of such eigenvalues. We will use the following  fact proved by  G.Borot and A.Guionnet  in \cite{BGII}: along the subsequences such that $\boldsymbol{N_\star} \mod \mathbb{Z}^{g+1}  \longrightarrow \kappa $ where $\kappa \in [0;1[^{g+1}$ and under ${\mathbb{P}}_{V,\mathcal{A}}^{N}$,  the vector $N(\boldsymbol{\lambda}) - \lfloor \boldsymbol{N_\star} \rfloor $ converges towards a random discrete Gaussian vector (not necessarily centered).

\begin{thm}\label{theoremedge2}
Let $0\leq h \leq g$, $i = [\lfloor \boldsymbol{N_\star} \rfloor]_{h-1} +1$ and $\Delta_{h}(\boldsymbol{\lambda}) = [\lfloor \boldsymbol{N_\star} \rfloor]_{h-1} - [N(\boldsymbol{\lambda})]_{h-1}$. Define 

\begin{equation*}
\xi_h (\boldsymbol{\lambda})= \mathbf{1}_{\Delta_{h}(\boldsymbol{\lambda}) \geq 0} \ \alpha_h^- + \mathbf{1}_{\Delta_{h}(\boldsymbol{\lambda}) < 0} \ \alpha_{h-1}^+ \ ,
\end{equation*}
\noindent where the expression above simplifies to $\alpha_0^-$ for $h=0$. Then along the subsequences $\boldsymbol{N_\star}\mod \mathbb{Z}^{g+1} \longrightarrow \kappa $ and under $\tilde{\mathbb{P}}_{V,\mathcal{A}}^{N}$

\begin{equation*}
\begin{split}
 \xi_h  &\xrightarrow{\ \ \ \mathcal{L} \ \ \ }  \mathbf{1}_{\Delta_{h , \kappa} \geq 0} \ \alpha_h^- + \mathbf{1}_{\Delta_{h, \kappa}< 0} \ \alpha_{h-1}^+ \ , \\
 N^{2/3} (\lambda_i - \xi_h) &\xrightarrow{\ \ \ \mathcal{L} \ \ \ } \mathbf{1}_{\Delta_{h , \kappa} \geq 0} \  \Lambda_{\Delta_{h , \kappa} + 1}  \  \Phi^h(-2) + \mathbf{1}_{\Delta_{h , \kappa} < 0} \   \Lambda_{-\Delta_{h , \kappa} }  \  \Phi^{h-1}(2) \ , 
\end{split}
\end{equation*}
\noindent where $(\Lambda_i)_i$ denote the eigenvalues  of $SAO_{\beta}$, $\Phi^{h}$ is a transport map introduced later and $\Delta_{h, \kappa}$ is a  discrete Gaussian random variable independent from $\Lambda$ if $ 1 \leq h \leq g$, and   equals to $0$ if $h=0$. 
\end{thm}
\noindent We could state a similar result  about the joint distribution of  $k$ consecutive eigenvalues as well. We note also that  using the transport methods of this paper, and adapting the  methods presented in \cite{FG} (notably Lemma 4.1 and the proof of Corollary 2.8), we could prove universality of the correlation functions in the bulk. This would require  rigidity estimate for the fixed filling fractions model introduced in the next section , which could be done as in \cite{bourgade2012bulk}, \cite{YLI}. As this universality result has already been proved in \cite{ShII}, we do not continue in this direction.\\

\noindent In order to study the fluctuations of the eigenvalues we place ourselves in the setting of the fixed filling fraction model introduced in \cite{BGII}, in which the number of eigenvalues in each cut is fixed. The idea is to construct an an approximate transport between our original measure, and a measure in which the interaction inbetween different cuts has been removed. This measure can then be written as a product measure and we can use the results proved for the one cut regime in \cite{BFG}. We will construct this map in the second section and show universality in the fixed filling fractions models  in Section 3. We will deduce from it the proofs of Theorems \ref{theoreme},  \ref{theoremedge} and \ref{theoremedge2} in the fourth section.

\section*{Acknowledgements} The author  would like to thank Alice Guionnet for the very helpful discussions and  comments.

\section{Fixed Filling Fractions}
\subsection{Introducing the model}
\noindent We consider a slightly different model with a more general type of interaction between  the particles and in which the number of particles in each cut is fixed. We will refer to \cite{BGK} for the known  results in this setting. For each  $ 0 \leq h \leq g$, let  $ B_h=[\beta_{h,-} ; \beta_{h,+}]$ be a small enlargement of $A_h=[\alpha_{h,-} ; \alpha_{h;+}]$ included in $U_h$ and  $ B = \underset {{0\leq h \leq g}}{\bigcup} B_h$ . It is well known (see for instance \cite{BGI}) that under our Hypothesis, the eigenvalues will leave $B$ with an exponentially small probability and we can thus study the behaviour of the eigenvalues under  $ \mathbb{P}_{V,B}^{N}$ instead of  $ \mathbb{P}_{V,\mathcal{A}}^{N}$ without loss of generality.  \\

\noindent We fix $\mathbf{N}=(N_0,\cdots,N_g) \in \mathbb{N}^{g+1}$ such that $\sum_{h=0}^{g} N_h = N$ and we want to consider a model in which the number of particles in each $B_h$ is fixed equal to $N_h$. Let $\boldsymbol\epsilon = \mathbf{N}/N \in ]0;1[^{g+1}$ and  for $T : B \times B \longrightarrow \mathbb{R}$ consider the probability measure on  $\mathbf{B}=\prod_{h=0}^g (B_h)^{N_h}$:

\begin{equation}
\begin{split}
\mathbb{P}_{T,B}^{N,\boldsymbol\epsilon} (\mathbf{d}\boldsymbol\lambda):=  \frac{1}{Z_{T,B}^{N,\boldsymbol\epsilon}} & \ \ \ \   \prod_{h=0}^g \prod_{1\leq i < j \leq N_h} \lvert \lambda_{h,i} - \lambda_{h,j} \rvert^\beta  \exp \Big( {- \frac{1}{2} \sum_{0\leq h, h' \leq g} \sum_{\substack{1\leq i \leq N_h \\ 1\leq j \leq N_{h'}}} T(\lambda_{h,i}, \lambda_{h',j})} \Big)\\
& \prod_{0\leq h < h' \leq g} \prod_{\substack{1\leq i \leq N_h \\ 1\leq j \leq N_{h'}}} \lvert \lambda_{h,i} - \lambda_{h',j} \rvert^\beta  \prod_{h=0}^g  \prod_{i=1}^{N_h} \mathbf{1}_{B_h}(\lambda_{h,i}) d\lambda_{h,i}  .
\end{split}
\end{equation}

\noindent  Note that with $T(\lambda_1, \lambda_2) = - (V(\lambda_1) +  V(\lambda_2))$  and without the location constraints, we are in the same setting as in the previous section.\\

\noindent As in the original model, we can prove the following result ( see \cite{BGK}):
\begin{prop}
Assume that $T : B \times B \longrightarrow \mathbb{R}$ is continuous.\\

\noindent Assume also that the energy defined by 
\begin{equation}\label{energyII}
E(\mu) = - \frac{1}{2}\int T(x_1,x_2) + \beta \log |x_1-x_2 | d\mu(x_1) d\mu(x_2)
\end{equation}

\noindent has a unique global minimum on the space $\mathcal{M}_1^{\boldsymbol\epsilon}(B)$ of probability measures on $B$ satisfying $\mu [B_h]= \epsilon_h$.\\

\noindent Then  under $\mathbb{P}_{T,B}^{N,\boldsymbol\epsilon}$ the normalized empirical measure $L_N = N^{-1} \sum_{h=0}^{g} \sum_{i=1}^{N_h} \delta_{\lambda_{h,i}}$ converges almost surely and in expectation towards the unique probability measure $\mu_T^{\boldsymbol\epsilon}$ which minimizes the energy.\\

\noindent Moreover it has compact support $A_T^{\boldsymbol\epsilon}$ and it is uniquely determined by the existence of  constants $C_{\boldsymbol\epsilon,h}$ such that:

\begin{equation}\label{caracterisation}
\beta \int_B\log |x-y|  d\mu_{T}^{\boldsymbol\epsilon}(y) + \int_B T(x,y) d\mu_{T}^{\boldsymbol\epsilon}(y)  \leq C_{\boldsymbol\epsilon,h}  \ \ \ on \ \  B_h
\end{equation}

\noindent with equality almost everywhere on the support.The support of  $\mu_{T}^{\boldsymbol\epsilon}$  is a union of $l+1$  intervals  $A_T^{\boldsymbol\epsilon} = \underset {{0\leq h \leq l}}{\bigcup} [\alpha^{T,\boldsymbol\epsilon}_{h,-} ; \alpha^{T,\boldsymbol\epsilon}_{h,+}]$  with $\alpha^{T,\boldsymbol\epsilon}_{h,-} < \alpha^{T,\boldsymbol\epsilon}_{h,+}$ , $ l \geq g$    and if $T$ is analytic on a neighbourhood of $A_T^{\boldsymbol\epsilon}$ ,
\begin{equation*}
\frac{d\mu_{T}^{\boldsymbol\epsilon}}{dx}= S_T^{\boldsymbol\epsilon}(x)\prod_{h=0}^l \sqrt{\lvert x - \alpha^{T,\boldsymbol\epsilon}_{h,-} \rvert \lvert x - \alpha^{T,\boldsymbol\epsilon}_{h,+} \rvert}  \ , 
\end{equation*}
\noindent with $S_T^{\boldsymbol\epsilon}$ analytic on a neighbourhood of $A_T^{\boldsymbol\epsilon}$.

\end{prop}
\noindent We point out the fact that the previous  theorem is also valid in the unconstrained case. In that case, we denote by $\mu_T$ the equilibrium measure.
\noindent Let $\boldsymbol\epsilon_{\star,T} = (\mu_{T}(B_h))_{0 \leq h \leq g} $. Then it is obvious that  $\mu_{T}^{\boldsymbol\epsilon_{\star,T}} = \mu_{T}$ . It is  shown in \cite{BGK} that we have the following:\\

\begin{lem}\label{regularity}
If $T$  extends to an analytic function on a neighbourhood of $B$ and the energy definied in (\ref{energyII}) has a unique minimizer over $\mathcal{M}_1(B)$ then for $\boldsymbol\epsilon$ close enough from $\boldsymbol\epsilon_\star$, the energy has a unique minimizer over $\mathcal{M}_1^{\boldsymbol\epsilon}(B)$ and  the number of cuts of the support of  $\mu_{T}^{\boldsymbol\epsilon}$ and $\mu_{T}$ are the same. Moreover,  $\alpha^{T,\boldsymbol\epsilon}_{h,-}$ , $\alpha^{T,\boldsymbol\epsilon}_{h,+}$ and $S_T^{\boldsymbol\epsilon}$ are smooth functions of  $\boldsymbol\epsilon$ (for the $L^\infty$ norm on $B$).
\end{lem}

\noindent They also prove a control of large  deviations of the largest eigenvalue under $\mathbb{P}_{T,B}^{N,\boldsymbol\epsilon}$.\\

\noindent We define the effective potential as  
\begin{equation}
\tilde{T}^{\boldsymbol\epsilon} (x)= \beta \int_B\log |x-y|  d\mu_{T}^{\boldsymbol\epsilon}(y) + \int_B T(x,y) d\mu_{T}^{\boldsymbol\epsilon}(y)  - C_{\boldsymbol\epsilon,h}  \ \ \ on \ \  B_h .
\end{equation}

\begin{lem}\label{GDP}
Let T satisfy the  conditions of the previous theorem. Then for any closed $F \subset B\setminus A_T^{\boldsymbol\epsilon}$  and open $O \subset B\setminus A_T^{\boldsymbol\epsilon}$  we have

\begin{equation*}
\left\{
\begin{split}
\limsup \frac{1}{N}\log \mathbb{P}_{T,B}^{N,\boldsymbol\epsilon} ( \exists i \  \lambda_i \in F) \leq &  \sup_{x \in F} \tilde{T}^{\boldsymbol\epsilon}(x) . \\
\liminf \frac{1}{N}\log \mathbb{P}_{T,B}^{N,\boldsymbol\epsilon} ( \exists i \ \lambda_i \in O) \geq &  \sup_{x \in O} \tilde{T}^{\boldsymbol\epsilon}(x) .
\end{split}
\right.
\end{equation*}
\end{lem}

\bigskip
\bigskip

\noindent We consider a potential $V$ on $\mathcal A$ satifying Hypothethis \ref{hypo} and  the potentials $T_0(x,y) = - (V(x) + V(y))$  and $T_1(x,y) = - (\tilde{V}^{\boldsymbol\epsilon} (x) + \tilde{V}^{\boldsymbol\epsilon} (y) + W(x,y))$ where 

\begin{equation*}
W(x,y)= 
\left\lbrace
\begin{split} 
\beta \log(x-y) & \ \ \  if \ \  x\in U_h \ , \ y\in U_{h'}  \ \ h> h'  \\
\beta \log(y-x) & \ \ \   if \ \  x\in U_h \ , \ y\in U_{h'} \ \  h< h' \\
0\ \ \ \  \  & \ \ \   if \ \  x\in U_h \ , \ y\in U_{h} 
\end{split} 
\right.
\end{equation*}

\noindent and
\begin{equation*}
\tilde{V}^{\boldsymbol\epsilon} (x) = V(x) -\int W(x,y) d\mu_{V}^{\boldsymbol\epsilon}(y) .
\end{equation*}

\bigskip
\noindent The key point is that $d \mathbb{P}_{T_1,B}^{N,\boldsymbol\epsilon}$ is a product measure as the interaction between cuts has been removed. Moreover, we can check by the characterization (\ref{caracterisation}) that 

\begin{equation*}
\mu_{V}^{\boldsymbol\epsilon} = \mu_{T_0}^{\boldsymbol\epsilon} = \mu_{T_1}^{\boldsymbol\epsilon} .
\end{equation*}

\noindent We now consider 
\begin{equation}\label{potentiel}
T_t = (1-t) T_0 + t \ T_1  \ \ , \ \ t\in [0;1]  .
\end{equation}

\noindent Still by (\ref{caracterisation}) we can check that for all $t \in [0;1]$ we have:
\begin{equation*}
\mu_{T_t}^{\boldsymbol\epsilon} =(1-t) \mu_{T_0}^{\boldsymbol\epsilon} + t \mu_{T_1}^{\boldsymbol\epsilon}= \mu_{V}^{\boldsymbol\epsilon}.
\end{equation*}
\\
\begin{rem}
Note that, by Lemma 2.2, for $\boldsymbol\epsilon$ in a small neighbourhood of  $\boldsymbol\epsilon_\star$ (that we will denote $\tilde{\mathcal{E}}$) the support $A^{\boldsymbol\epsilon}$ of $\mu_{T_t}^{\boldsymbol\epsilon} =\mu_{V}^{\boldsymbol\epsilon}$ has $g+1$ cut  and we can write
\begin{equation}
d\mu_{T_t}^{\boldsymbol\epsilon}=d\mu_{V}^{\boldsymbol\epsilon}= S^{\boldsymbol\epsilon}(x)\prod_{ h=0}^g  \sqrt{\lvert x - \alpha^{\boldsymbol\epsilon}_{h,-} \rvert \lvert x - \alpha^{\boldsymbol\epsilon}_{h,+} \rvert}   dx \ ,
\end{equation}

\noindent with $S^{\boldsymbol\epsilon}$ positive on $A^{\boldsymbol\epsilon}$.
\end{rem}

\begin{rem}\label{controle}
Note also that by the last point of Hypothesis \ref{hypo} and by Lemma 2.2, if we fix a closed interval $F \subset B\setminus A$, then   for $\boldsymbol\epsilon$ close enough to $\boldsymbol\epsilon_\star$ and all $t \in [0;1]$ ,  $\tilde{T_t}^{\boldsymbol\epsilon} < 0 $ on $F$.

\end{rem}
\noindent The goal is to build first an approximate transport map   between the measures $d \mathbb{P}_{T_t,B}^{N,\boldsymbol\epsilon}$ for a fixed $\boldsymbol\epsilon$ in $\tilde{\mathcal{E}}$ i.e find a map $X_1^{N,\boldsymbol\epsilon}$ that satisfies for all $f: \mathbb{R}^N \longrightarrow \mathbb{R}$  bounded measurable function \\

\begin{equation}\label{almosttransp}
\lvert \int f(X_1^{N,\boldsymbol\epsilon})\ d \mathbb{P}_{V,B}^{N,\boldsymbol\epsilon} - \int f \  d \mathbb{P}_{T_1,B}^{N,\boldsymbol\epsilon} \rvert \  \leq C \norm{f}_{\infty}  \frac{(\log N)^3}{N} .
\end{equation}

\noindent We will see that we can   build a  transport map depending smoothly on $\epsilon$ and show universality  in the fixed filling model. We will then use this result to prove universality in the original model.\\

\begin{prop}\label{theoremfixed}
Assume that $V$ satisfies Hypothesis \ref{hypo}, and that $T_t$ is  as defined previously. Let $\mathbf{N}=(N_0,\cdots,N_g)$ such that  $\boldsymbol\epsilon = \mathbf{N}/N$ is in $\tilde{\mathcal{E}}$  and $\tilde{\mathbb{P}}_{T,B}^{N,\boldsymbol\epsilon}$ denote the distribution of the ordered eigenvalues under  $ \mathbb{P}_{T,B}^{N,\boldsymbol\epsilon}$.  Then for a constant C independent of ${\boldsymbol\epsilon}$ and $N$, and if for all  $0 \leq h \leq g$ $f_h: \mathbb{R}^m\longrightarrow \mathbb{R}$ is Lipschitz supported inside $[-M , M]^{m}$  we have:
\begin{enumerate}
\item Eigenvalue gaps in the Bulk \\

\begin{equation*}
\begin{split}
\Bigg|  \int \prod_{0 \leq h \leq g} f_h\big(N(\lambda_{h,i_h+1}-& \lambda_{h,i_h}),  \cdots,N(\lambda_{h,i_h+m} - \lambda_{h,i_h})\big)d\tilde{\mathbb{P}}_{V,B}^{N,\boldsymbol\epsilon} \\
   -  & \int \prod_{0 \leq h \leq g} f_h\big(N (\lambda_{h,i_h+1}-\lambda_{h,i_h}),\cdots,N (\lambda_{h,i_h+m} - \lambda_{h,i_h})\big)d\tilde{\mathbb{P}}_{T_1,B}^{N,\boldsymbol\epsilon} \Bigg| \\
  \leq & C \frac{(\log N)^3}{N}\norm{f}_\infty + C(\sqrt{m}\frac{(\log N)^2}{N^{1/2}} + M \frac{(\log N)}{N^{1/2}} ) \norm{\nabla f}_\infty 
\end{split}
\end{equation*}

\item Eigenvalue gaps at the Edge

\begin{equation*}
\begin{split}
\Bigg|  \int \prod_{0 \leq h \leq g} f_h\big(N^{2/3}(\lambda_{h,1}- & \alpha_{h,-}^{\boldsymbol\epsilon}),  \cdots,N^{2/3}(\lambda_{h,m} - \alpha_{h,-}^{\boldsymbol\epsilon})\big)d\tilde{\mathbb{P}}_{V,{B}}^{N,\boldsymbol\epsilon} \\
   -  & \int \prod_{0 \leq h \leq g} f_h\big(N^{2/3} (\lambda_{h,1}-\alpha_{h,-}^{\boldsymbol\epsilon}),\cdots,N^{2/3} (\lambda_{h,m} - \alpha_{h,-}^{\boldsymbol\epsilon})\big)d\tilde{\mathbb{P}}_{T_1,B}^{N,\boldsymbol\epsilon} \Bigg| \\
  \leq & C \frac{(\log N)^3}{N}\norm{f}_\infty + C(\sqrt{m}\frac{(\log N)^2}{N^{5/6}}  +\frac{\log N}{N^{1/3}}) \norm{\nabla f}_\infty
\end{split}
\end{equation*}

\end{enumerate}

\noindent where we defined $f: \mathbb{R}^{m (g+1)}\longrightarrow \mathbb{R}$ by $f(\boldsymbol{x}_{0},\cdots,\boldsymbol{x}_{g}) = \prod_{0 \leq h \leq g} f_h(\boldsymbol{x}_h)$.
\end{prop}

\noindent We  deduce the following corollary from the results obtained in the one-cut regime in \cite{BFG}, and from the fact that $\tilde{\mathbb{P}}_{T_1,B}^{N,\boldsymbol\epsilon} $ is a product measure.

\begin{cor}\label{corfixed} Assume the same hypothesis as in the precedent proposition. We write $\mu_{V}^{\boldsymbol\epsilon} = \sum_{0 \leq h \leq g} \epsilon_h \  \mu_{V}^{\boldsymbol\epsilon,h}$ where $\mu_{V}^{\boldsymbol\epsilon,h}$ has connected support .   For   some transport maps $\Phi^{\boldsymbol\epsilon,h }$  from $\mu_G$   to $\mu_{V}^{\boldsymbol\epsilon,h}$, \\
\begin{enumerate}
\item Eigenvalue gaps in the Bulk

\begin{equation*}
\begin{split}
\Bigg|  \int \prod_{0 \leq h \leq g}& f_h\big(N(\lambda_{h,i_h+1}- \lambda_{h,i_h}),  \cdots,N(\lambda_{h,i_h+m} - \lambda_{h,i_h})\big)d\tilde{\mathbb{P}}_{V,B}^{N,\boldsymbol\epsilon} \\
   -  & \prod_{0 \leq h \leq g} \Bigg( \int  f_h\big(N (\Phi^{\boldsymbol\epsilon,h })'(\lambda_{i_h})(\lambda_{i_h+1}-\lambda_{i_h}),\cdots,N (\Phi^{\boldsymbol\epsilon,h })'(\lambda_{i_h})(\lambda_{i_h+m} - \lambda_{i_h})\big)d\tilde{\mathbb{P}}_{G}^{N_h} \Bigg) \Bigg| \\
  \leq & C \frac{(\log N)^3}{N}\norm{f}_\infty + C(\sqrt{m}\frac{(\log N)^2}{N^{1/2}} + M \frac{(\log N)}{N^{1/2}} +\frac{M^2}{N} ) \norm{\nabla f}_\infty
\end{split}
\end{equation*}

\item Eigenvalue gaps  at the Edge
\begin{equation*}
\begin{split}
\Bigg|  \int & \prod_{0 \leq h \leq g}  f_h\big(N^{2/3}(\lambda_{h,1}-  \alpha_{h,-}^{\boldsymbol\epsilon}),  \cdots,N^{2/3}(\lambda_{h,m} - \alpha_{h,-}^{\boldsymbol\epsilon})\big)d\tilde{\mathbb{P}}_{V,B}^{N,\boldsymbol\epsilon} \\
   -  & \prod_{0 \leq h \leq g} \Bigg( \int  f_h\big(N^{2/3} (\Phi^{\boldsymbol\epsilon ,h })'(-2)(\lambda_{1}+2),\cdots,N^{2/3} (\Phi^{\boldsymbol\epsilon,h })'(-2)(\lambda_{m} + 2)\big)d\tilde{\mathbb{P}}_{G}^{N_h} \Bigg) \Bigg| \\
  \leq & C \frac{(\log N)^3}{N}\norm{f}_\infty + C(\sqrt{m}\frac{(\log N)^2}{N^{5/6}}  +\frac{\log N}{N^{1/3}} +  \frac{M^2}{N^{4/3}}) \norm{\nabla f}_\infty .
\end{split}
\end{equation*}

 \end{enumerate}
\end{cor}

\noindent The proof of the theorem will be similar to what has already been done in the one-cut case, one major difference being the inversion of the operator $\Xi$ introduced in Lemma 3.2 of \cite{BFG}.

\subsection{Approximate Monge Ampère Equation}

\noindent The analysis done in the one-cut regime suggests   to look at  the transport as  the flow of an approximate solution to the Monge Ampère equation $\mathbf{Y}_t^{N,\boldsymbol\epsilon} = ( \mathbf{Y}_{0,t}^{N,\boldsymbol\epsilon} , \cdots , \mathbf{Y}_{g,t}^{N,\boldsymbol\epsilon} ) :\mathbb{R}^N \longrightarrow \mathbb{R}^N$ where $\mathbf{Y}_{h,t}^{N,\boldsymbol\epsilon}:\mathbb{R}^N \longrightarrow \mathbb{R}^{N_h} $ solves the following equation:
\begin{equation}\label{approx}
\begin{split}
& \divergence { (\mathbf{Y}_t^{N,\boldsymbol\epsilon})} =  c_t^{N,\boldsymbol\epsilon} - \beta \sum_{h=0}^g  \sum_{1\leq i < j \leq N_h}  \frac{\mathbf{Y}^{N,\boldsymbol\epsilon}_{h,i,t} - \mathbf{Y}^{N,\boldsymbol\epsilon}_{h,j,t}}{\lambda_{h,i} - \lambda_{h,j}} - \beta \sum_{0\leq h < h' \leq g} \sum_{\substack{1\leq i \leq N_h \\ 1\leq j \leq N_{h'}}} \frac{\mathbf{Y}^{N,\boldsymbol\epsilon}_{h,i,t} - \mathbf{Y}^{N,\boldsymbol\epsilon}_{h',j,t}}{\lambda_{h,i} - \lambda_{h',j}}  \\
& -  {\sum_{0\leq h, h' \leq g} \sum_{\substack{1\leq i \leq N_h \\ 1\leq j \leq N_{h'}}} } (\partial_1 T(\lambda_{h,i}, \lambda_{h',j})\mathbf{Y}^{N,\boldsymbol\epsilon}_{h,i,t}  - \frac{1}{2} W(\lambda_{h,i}, \lambda_{h',j})) -N  \sum_{0\leq h \leq g} \sum_{{1\leq i \leq N_h }} \int W(\lambda_{h,i},z)d\mu_{V}^{\boldsymbol\epsilon}(z) 
\end{split}
\end{equation}

\noindent where
\begin{equation*}
\begin{split}
c_t^{N,\boldsymbol\epsilon} & =   \int \Bigg( N  \sum_{0\leq h \leq g} \sum_{{1\leq i \leq N_h }} \int W(\lambda_{h,i},z)d\mu_{V}^{\boldsymbol\epsilon}(z) - \frac{1}{2} \sum_{0\leq h, h' \leq g} \sum_{\substack{1\leq i \leq N_h \\ 1\leq j \leq N_{h'}}} W(\lambda_{h,i} ,\lambda_{h',j})  \ \Bigg) d \mathbb{P}_{T_t,B}^{N,\boldsymbol\epsilon} (\boldsymbol\lambda) \\
& = \partial_t \log ( Z_{V_t,B}^{N,\boldsymbol\epsilon}).
\end{split}
\end{equation*}
\noindent Let $\mathcal{R}_t^{N,\boldsymbol\epsilon}(\mathbf{Y}^{N,\boldsymbol\epsilon})$ the error term defined as
\begin{equation}
\begin{split}
& \mathcal{R}_t^{N,\boldsymbol\epsilon}(\mathbf{Y}^{N,\boldsymbol\epsilon})=    \beta \sum_{h=0}^g  \sum_{1\leq i < j \leq N_h}  \frac{\mathbf{Y}^{N,\boldsymbol\epsilon}_{h,i,t} - \mathbf{Y}^{N,\boldsymbol\epsilon}_{h,j,t}}{\lambda_{h,i} - \lambda_{h,j}} + \beta \sum_{0\leq h < h' \leq g} \sum_{\substack{1\leq i \leq N_h \\ 1\leq j \leq N_{h'}}} \frac{\mathbf{Y}^{N,\boldsymbol\epsilon}_{h,i,t} - \mathbf{Y}^{N,\boldsymbol\epsilon}_{h',j,t}}{\lambda_{h,i} - \lambda_{h',j}}  \\
& + {\sum_{0\leq h, h' \leq g} \sum_{\substack{1\leq i \leq N_h \\ 1\leq j \leq N_{h'}}} } (\partial_1 T(\lambda_{h,i}, \lambda_{h',j})\mathbf{Y}^{N,\boldsymbol\epsilon}_{h,i,t} - \frac{1}{2} W(\lambda_{h,i}, \lambda_{h',j})) + N  \sum_{0\leq h \leq g} \sum_{{1\leq i \leq N_h }} \int W(\lambda_{h,i},z)d\mu_{V}^{\boldsymbol\epsilon}(z) \\
&  + \divergence { (\mathbf{Y}_t^{N,\boldsymbol\epsilon})} - c_t^{N,\boldsymbol\epsilon} .
\end{split}
\end{equation}

\noindent We have the following stability lemma
\begin{lem}\label{lemmatransp}
Let  $\mathbf{Y}_t^{N,\boldsymbol\epsilon} :\mathbb{R}^N \longrightarrow \mathbb{R}^N$ be a smooth vector field and let $X^{N,\boldsymbol\epsilon}$ be its flow: 
\begin{equation}
\dot{X}_t^{N,\boldsymbol\epsilon} = \mathbf{Y}_t^{N,\boldsymbol\epsilon} (X_t^{N,\boldsymbol\epsilon})  \ \ \ \ X_0^{N,\boldsymbol\epsilon} =  Id .
\end{equation}

\noindent Assume that  $\mathbf{Y}_t^{N,\boldsymbol\epsilon}$ vanishes on the boundary of $\mathbf{B}$.

\noindent Let $f: \mathbb{R}^N \longrightarrow \mathbb{R}$ be a bounded measurable function. Then
\begin{equation*}
\Big| \int f(X_t^{N,\boldsymbol\epsilon})\ d \mathbb{P}_{V,B}^{N,\boldsymbol\epsilon} - \int f \  d \mathbb{P}_{T_t,B}^{N,\boldsymbol\epsilon} \Big|  \leq \norm{f}_{\infty}  \int_0^t \norm{\mathcal{R}_s^{N,\boldsymbol\epsilon}(\mathbf{Y}^{N,\boldsymbol\epsilon})}_{L^1( \mathbb{P}_{T_s,B}^{N,\boldsymbol\epsilon})} ds .
\end{equation*}
\end{lem}

\begin{proof}

Let
\begin{equation*}
\begin{split}
\rho_t (\boldsymbol{\lambda}):=  \frac{1}{Z_{T,B}^{N,\boldsymbol\epsilon}}   \ & \ \ \    \prod_{h=0}^g \prod_{1\leq i < j \leq N_h} \lvert \lambda_{h,i} - \lambda_{h,j} \rvert^\beta  \exp \Big( {- \frac{1}{2} \sum_{0\leq h, h' \leq g} \sum_{\substack{1\leq i \leq N_h \\ 1\leq j \leq N_{h'}}} T_t(\lambda_{i,h}, \lambda_{j,h'})} \Big) \\
 & \prod_{0\leq h < h' \leq g} \prod_{\substack{1\leq i \leq N_h \\ 1\leq j \leq N_{h'}}} \lvert \lambda_{h,i} - \lambda_{h',j} \rvert^\beta 
\end{split}
\end{equation*}

\noindent and $JX_t^{N,\boldsymbol\epsilon}$ denote the Jacobian of $X_t^{N,\boldsymbol\epsilon}$.  As $\mathbf{Y}_t^{N,\boldsymbol\epsilon}$ vanishes on the boundary of $\mathbf{B}$  , $X_{t}^{N,\boldsymbol\epsilon}(\mathbf{B}) = \mathbf{B}$.  By the change of variable formula we have 

\begin{equation*}
\begin{split}
\int f \  d \mathbb{P}_{T_t,B}^{N,\boldsymbol\epsilon} & =   \int_{\mathbf{B}} f(\boldsymbol{\lambda})\rho_t(\boldsymbol{\lambda}) \boldsymbol{d}\boldsymbol{\lambda}  = \int_{X_t^{N,\boldsymbol\epsilon}(\mathbf{B})} f(\boldsymbol{\lambda})\rho_t(\boldsymbol{\lambda}) \boldsymbol{d}\boldsymbol{\lambda}  \\
&= \int_{\mathbf{B}} f(X_t^{N,\boldsymbol\epsilon}) \rho_t(X_t^{N,\boldsymbol\epsilon}) JX_t^{N,\boldsymbol\epsilon} \boldsymbol{d}\boldsymbol{\lambda}
\end{split} .
\end{equation*}

\noindent Thus we have 
\begin{equation*}
\lvert \int f(X_t^{N,\boldsymbol\epsilon})\ d \mathbb{P}_{T_0,B}^{N,\boldsymbol\epsilon} - \int f \  d \mathbb{P}_{T_t,B}^{N,\boldsymbol\epsilon} \rvert \leq \norm{f}_{\infty} \int_{\mathbf{B}} \lvert \rho_0(\boldsymbol{\lambda}) - \rho_t(X_t^{N,\boldsymbol\epsilon}) JX_t^{N,\boldsymbol\epsilon} \rvert \boldsymbol{d}\boldsymbol{\lambda} .
\end{equation*}

\noindent Let 

\begin{equation*}
\Delta_t = \partial_t \int_{\mathbf{B}} \lvert \rho_0(\boldsymbol{\lambda}) - \rho_t(X_t^{N,\boldsymbol\epsilon}) JX_t^{N,\boldsymbol\epsilon} \rvert \boldsymbol{d}\boldsymbol{\lambda} .
\end{equation*}

\noindent  Using $\partial_t(JX_t^{N,\boldsymbol\epsilon}) = \divergence { (\mathbf{Y}_t^{N,\boldsymbol\epsilon})} JX_t^{N,\boldsymbol\epsilon}$ we have 

\begin{equation*}
\begin{split}
\Delta_t \leq & \int_{\mathbf{B}} \lvert \partial_t \left(JX_t^{N,\boldsymbol\epsilon}  \rho_t(X_t^{N,\boldsymbol\epsilon}) \right) \rvert \boldsymbol{d}\boldsymbol{\lambda}  \\
=& \int_{\mathbf{B}} \lvert \divergence { (\mathbf{Y}_t^{N,\boldsymbol\epsilon})}  JX_t^{N,\boldsymbol\epsilon}  \rho_t(X_t^{N,\boldsymbol\epsilon}) +  JX_t^{N,\boldsymbol\epsilon}  (\partial_t\rho_t)(X_t^{N,\boldsymbol\epsilon}) +  JX_t^{N,\boldsymbol\epsilon}  \nabla\rho_t(X_t^{N,\boldsymbol\epsilon})\dot{X}_t^{N,\boldsymbol\epsilon}   \rvert \boldsymbol{d}\boldsymbol{\lambda} \\
=& \int \lvert \mathcal{R}_t^{N,\boldsymbol\epsilon}(\mathbf{Y}^{N,\boldsymbol\epsilon}) \rvert d \mathbb{P}_{T_t,B}^{N,\boldsymbol\epsilon}
\end{split}
\end{equation*}

\noindent and this gives the lemma. \\
\end{proof}

\subsection{Constructing an Approximate Solution}

\noindent The construction of the approximate solution will be very similar to Section 3 of \cite{BFG}.\\

\noindent We fix $t \in [0;1]$ ,  $\mathbf{N}=(N_0,\cdots,N_g) \in \mathbb{N}^{g+1}$ such that $\sum_{h=0}^{g} N_h = N$ and set $\boldsymbol\epsilon = \mathbf{N}/N \in ]0;1[^{g+1}$ .\\

\noindent Let
\begin{equation*}
L_N =\frac{1}{N}  \sum_{h,i} \delta_{\lambda_{h,i}} \ \ ,  \  \ M_N = \sum_{h,i} \delta_{\lambda_{h,i}} - N \mu_{V}^{\boldsymbol\epsilon}  .
\end{equation*}

\noindent We look for a map $\mathbf{Y}_t^{N,\boldsymbol\epsilon} = (\mathbf{Y}_{0,1,t}^{N,\boldsymbol\epsilon} , \cdots , \mathbf{Y}_{g,N_g,t}^{N,\boldsymbol\epsilon} ) : \mathbb{R}^N \longrightarrow \mathbb{R}^N$ approximately solving (\ref{approx}). As in the one-cut regime, we make the following ansatz:

\begin{equation}\label{expansion}
\mathbf{Y}_{h,i,t}^{N,\boldsymbol\epsilon} (\boldsymbol{\lambda}) =  \frac{1}{N} \mathbf{y}_{1,t}^{\boldsymbol\epsilon} (\lambda_{h,i}) +  \frac{1}{N} \mathbf{\xi}_{t}^{\boldsymbol\epsilon} (\lambda_{h,i}, M_N)  \ , \ \mathbf{\xi}_{t}^{\boldsymbol\epsilon} (x, M_N) = \int \mathbf{z}_{t}^{\boldsymbol\epsilon} (x,y)dM_N(y)
\end{equation}

\noindent for some functions $\mathbf{y}_{1,t}^{\boldsymbol\epsilon} : \mathbb{R} \longrightarrow \mathbb{R}$ and $\mathbf{z}_{t}^{\boldsymbol\epsilon} : \mathbb{R}^2 \longrightarrow \mathbb{R}$.\\

\begin{prop}\label{erreur}
Let $V$ satisfy Hypothesis \ref{hypo} and  $T_t$ is as in (\ref{potentiel}) . Then there are  $\mathbf{y}_{1,t}^{\boldsymbol\epsilon}$ in $C^\infty(\mathbb{R})$ and $ \mathbf{z}_{t}^{\boldsymbol\epsilon}$ in $C^\infty(\mathbb{R}^2)$ such that for a constant C, for all $t\in [0;1]$ and $\boldsymbol\epsilon \in \tilde{\mathcal{E}}$: 
\begin{equation*}
\norm{\mathcal{R}_t^{N,\boldsymbol\epsilon}(\mathbf{Y}^{N,\boldsymbol\epsilon})}_{L^1( \mathbb{P}_{T_t,B}^{N,\boldsymbol\epsilon})} \leq C \frac{(\log N)^3}{N}.
\end{equation*}
\end{prop}

\noindent Using the substitution (\ref{expansion}), we have to find equations for  $\mathbf{y}_{1,t}^{\boldsymbol\epsilon}$ and $ \mathbf{z}_{t}^{\boldsymbol\epsilon} $. To simplify the notations,we will write $\mathcal{R}$ instead of $\mathcal{R}_t^{N,\boldsymbol\epsilon}(\mathbf{Y}^{N,\boldsymbol\epsilon})$. We obtain:

\begin{equation*}
\begin{split}
\mathcal{R} & =  - \frac{N^2}{2} \iint W dL_N dL_N  + N^2 \int W dL_N d\mu_{V}^{\boldsymbol\epsilon}  \\
& + \frac{\beta N}{2} \iint \frac{\mathbf{y}^{\boldsymbol\epsilon}_{1,t}(x) - \mathbf{y}^{\boldsymbol\epsilon}_{1,t}(y)}{x - y} dL_N(x) dL_N(y) + N \int \partial_1 T_t(x,y) \mathbf{y}^{ \boldsymbol\epsilon}_{1,t}(x) dL_N(x) dL_N(y)\\
& + \frac{\beta N}{2} \iint \frac{\mathbf{\xi}^{\boldsymbol\epsilon}_{t}(x,M_N) - \mathbf{\xi}^{\boldsymbol\epsilon}_{t}(y,M_N)}{x - y} dL_N(x) dL_N(y) + N \int \partial_1 T_t(x,y) \mathbf{\xi}^{ \boldsymbol\epsilon}_{t}(x,M_N) dL_N(x) dL_N(y) \\
& +  \frac{1}{N}\mathbf{\eta}(M_N) +  \left(1-\frac{\beta}{2}\right)\int {\mathbf{y}^{\boldsymbol\epsilon}_{1,t}}' dL_N + \left(1-\frac{\beta}{2}\right) \int \partial_1\mathbf{\xi}^{\boldsymbol\epsilon}_t(x,M_N) dL_N(x)  + \tilde{c}_t^N
\end{split}
\end{equation*}

\noindent where $\tilde{c}_t^N$ is a constant and for any measure $\nu$ we set
\begin{equation*}
\mathbf{\eta}(\nu) = \int \partial_2\mathbf{z}^{\boldsymbol\epsilon}_t(y,y) d\nu(y).
\end{equation*}

\noindent We use equilibrium relations to recenter $L_N$ by $\mu_{V}^{\boldsymbol\epsilon}$. Consider $f$ a bounded measurable function on $B$ and  $\mu_{V,\delta}^{\boldsymbol\epsilon} = (x + \delta f(x))\# \mu_{V}^{\boldsymbol\epsilon}$. Then as for $\delta$ small enough $\mu_{V,\delta}^{\boldsymbol\epsilon} (B_h) = \epsilon_h$ for all $ 0 \leq h \leq g$, we have $E(\mu_{V,\delta}^{\boldsymbol\epsilon})  \geq E(\mu_{V}^{\boldsymbol\epsilon})$ where we defined the energy in (\ref{energy}). By differentiating at $\delta = 0$  we obtain
\begin{equation}\label{equilibre}
\frac{\beta}{2} \iint \frac{f(x)-f(y)}{x-y} d\mu_{V}^{\boldsymbol\epsilon}(x)d\mu_{V}^{\boldsymbol\epsilon}(y) +  \int \partial_1 T_t(x,y) f(x) d\mu_{V}^{\boldsymbol\epsilon}(x)d\mu_{V}^{\boldsymbol\epsilon}(y)=0 .
\end{equation} 

\noindent Thus, if we define the operator $\Xi$ acting  on  smooth functions $f:B \longrightarrow \mathbb{R}$ by
\begin{equation*}
\Xi f(x) =  \int \left[ \beta \  \frac{f(x)-f(y)}{x-y}  + \partial_1 T_t(x,y) f(x) + \partial_2 T_t(x,y) f(y) \right]d\mu_{V}^{\boldsymbol\epsilon}(y) \ ,
\end{equation*}

\noindent we obtain
\begin{equation*}
\begin{split}
 \frac{\beta }{2} & \iint \frac{f(x)-f(y)}{x-y}dL_N(x) dL_N(y) + \int \partial_1 T_t(x,y) f(x) dL_N(x) dL_N(y) \\
& =  \frac{1}{N} \int \Xi f dM_N +\frac{1}{N^2} \left[ \frac{\beta}{2} \iint \frac{f(x)-f(y)}{x-y} dM_N(x) dM_N(y) + \iint \partial_1 T_t(x,y)f(x) dM_N(x) dM_N(y) \right] \ .
\end{split}
\end{equation*}

\noindent Therefore we can write
\begin{equation*}
\begin{split}
\mathcal{R} &=  \int \left[ \Xi  \mathbf{y}^{\boldsymbol\epsilon}_{1,t} +\left(1- \frac{\beta}{2} \right)  \int \partial_1 \mathbf{z}^{\boldsymbol\epsilon}_{t}(z,\cdot)d\mu_{V}^{\boldsymbol\epsilon}(z) \right] dM_N \\
&+ \iint \left[ \Xi\mathbf{z}^{\boldsymbol\epsilon}_{t}(\cdot , y)[x] - \frac{1}{2}W(x,y) \right] dM_N(x) dM_N(y) + C_t^{N,\boldsymbol\epsilon} + E
\end{split}
\end{equation*}
\noindent with

\begin{equation*}
\Xi \mathbf{z}^{\boldsymbol\epsilon}_{t}(\cdot , y)[x]=  \int \left[ \beta \frac{\mathbf{z}^{\boldsymbol\epsilon}_{t}(x,y)-\mathbf{z}^{\boldsymbol\epsilon}_{t}(z,y)}{x-z}   + \partial_1 T_t(x,z) \mathbf{z}^{\boldsymbol\epsilon}_{t}(x,y) + \partial_2 T_t(x,z) \mathbf{z}^{\boldsymbol\epsilon}_{t}(z,y) \right] d\mu_{V}^{\boldsymbol\epsilon}(z)
\end{equation*}
\noindent where $C_t^{N,\boldsymbol\epsilon}$ is deterministic and $E$ is an error term:
\begin{equation}
\begin{split}
E &= \frac{1}{N} \int \partial_2 \mathbf{z}^{\boldsymbol\epsilon}_{t}(x,x) dM_N(x) + \frac{1}{N}\left(1-\frac{\beta}{2}\right)\int {\mathbf{y}^{\boldsymbol\epsilon}_{1,t}}' dM_N \\
&+  \frac{1}{N}\left(1-\frac{\beta}{2}\right) \iint \partial_1 \mathbf{z}^{\boldsymbol\epsilon}_{t}(x,y) dM_N(x)dM_N(y) \\
&+\frac{1}{N} \iint \left[ \frac{\beta}{2} \frac{\mathbf{y}^{\boldsymbol\epsilon}_{1,t}(x)-\mathbf{y}^{\boldsymbol\epsilon}_{1,t}(y)}{x-y} +\partial_1 T_t(x,y)\mathbf{y}^{\boldsymbol\epsilon}_{1,t}(x) \right] dM_N(x)dM_N(y) \\
&+\frac{1}{N} \iiint \left[ \frac{\beta}{2} \frac{\mathbf{z}^{\boldsymbol\epsilon}_{t}(x,y)-\mathbf{z}^{\boldsymbol\epsilon}_{t}(z,y)}{x-z} +\partial_1 T_t(x,z)\mathbf{z}^{\boldsymbol\epsilon}_{t}(x,y) \right] dM_N(x)dM_N(y)dM_N(z)
\end{split}.
\end{equation}

\noindent To make $\mathcal{R}$ small we need 
\begin{equation*}
\left\{
\begin{split}
\Xi & \mathbf{z}^{\boldsymbol\epsilon}_{t}(\cdot , y)[x] =  \frac{1}{2} W(x,y)+\kappa_1(x , y) \ , \\
\Xi & \mathbf{y}_{1,t}^{\boldsymbol\epsilon} =  \left(\frac{\beta}{2}-1 \right)  \int \partial_1 \mathbf{z}^{\boldsymbol\epsilon}_{t}(z,\cdot)d\mu_{V}^{\boldsymbol\epsilon}(z)  +\kappa_2 \ ,
\end{split}
\right.
\end{equation*}

\noindent where $ \kappa_2$ and $\kappa_1(\cdot , y)$ are functions on $B$ constant on each $B_h$.\\

\noindent The following lemma shows how to invert $\Xi$ and will give us the desired functions. We will denote by $\mathcal{O}(U)$ the set of holomorphic functions on $U$.

\begin{lem}\label{thm:inv}
Let $V$ satisfy Hypothesis \ref{hypo} , $T_t$ as in (\ref{potentiel}) and $\boldsymbol\epsilon = \mathbf{N}/N$ in $\tilde{\mathcal{E}}$. The support of  $\mu_{V}^{\boldsymbol\epsilon}$  is a union of $g+1$  intervals  $A^{\boldsymbol\epsilon} = \underset {{0\leq h \leq g}}{\bigcup} [\alpha^{\boldsymbol\epsilon}_{h,-} ; \alpha^{\boldsymbol\epsilon}_{h,+}]$  with $\alpha^{\boldsymbol\epsilon}_{h,-} < \alpha^{\boldsymbol\epsilon}_{h,+}$  and,
\begin{equation*}
\frac{d\mu_{V}^{\boldsymbol\epsilon}}{dx}= S(x)\prod \sqrt{\lvert x - \alpha^{\boldsymbol\epsilon}_{h,-} \rvert \lvert x - \alpha^{\boldsymbol\epsilon}_{h,+} \rvert}  
\end{equation*}
\noindent with $S$ positive on $A^{\boldsymbol\epsilon}$.

\noindent Let $k \in \mathcal{O}(U) $  and set for $f \in \mathcal{O}(U) $ 
\begin{equation*}
\Xi f(x) =  \int \left[ \beta \  \frac{f(x)-f(y)}{x-y}  + \partial_1 T_t(x,y) f(x) + \partial_2 T_t(x,y) f(y) \right]d\mu_{T_t}^{\boldsymbol\epsilon}(y) \ \ \ \ \ \forall x\in U.
\end{equation*}

\noindent Then there exists a unique function $\kappa_k$ on $U$ constant on each $U_h$ such that the equation

\begin{equation*}
\Xi f = k + \kappa_k
\end{equation*}

\noindent has a solution in $\mathcal{O}(U)$ . Moreover, for all $x\in U_h$ 
\begin{equation}\label{formuleinverse}
f(x) = - \frac{1}{2 \beta \pi^2 \sigma(x) \sigma_h(x) S(x)} \left[ \oint \frac{i \sigma_h(\xi) (k(\xi)+c_h^1)}{(\xi - x) } d\xi + c_h^2\right] \ , 
\end{equation}

\noindent where the contour surrounds x and  $A^{\boldsymbol\epsilon}_h$  in $U_h$ and 
\begin{equation*}
\begin{split}
\sigma^2(x) & = \prod_{h'} \left( x - \alpha^{\boldsymbol\epsilon}_{h',-}\right)\left(x  - \alpha^{\boldsymbol\epsilon}_{h',+}\right) \\
\sigma(x) & \underset{x \rightarrow \infty}{\sim} x^{g+1} \\
\sigma^2_h(x) & = \left( x - \alpha^{\boldsymbol\epsilon}_{h,-}\right)\left(x  - \alpha^{\boldsymbol\epsilon}_{h,+}\right) \\
\sigma_h(x) & \underset{x \rightarrow \infty}{\sim} x
\end{split}
\end{equation*}

\noindent and the constants $c_h^1$ and  $c_h^2$ are chosen in a way such that the expression under the bracket vanishes at  $x=\alpha^{\boldsymbol\epsilon}_{h,-}$ and $x=\alpha^{\boldsymbol\epsilon}_{h,+}$ for each $h$ (see the following  Lemma).\\

\noindent Moreover $f$ satisfies for all $ j$
\begin{equation}\label{norme}
\norm{f}_{C^j(B)} \leq C_j \norm{k}_{C^{j+2}(B)}
\end{equation}

\noindent for some constants $C_j$. We will denote $f$ by $\Xi^{-1} k$.

\end{lem}

\noindent Before proving this lemma we need another  lemma

\begin{lem}\label{surj}
Let $V \in \mathcal{O}(U)$ and $\mu_{V}^{\boldsymbol\epsilon}$  as in the previous lemma.\\

\noindent Then for all $ 0 \leq h \leq g$ the linear operator 
\begin{equation*}
\begin{split}
\boldsymbol{\Theta_h}:= \mathbb{C}^2 & \longrightarrow \mathbb{C}^{2} \\
(c^1,c^2) & \longrightarrow \left( c^1 \oint \frac{\sigma_h(\xi) }{(\xi - \alpha^{\boldsymbol\epsilon}_{h,-}) } d\xi + c^2 , c^1 \oint \frac{\sigma_h(\xi) }{(\xi - \alpha^{\boldsymbol\epsilon}_{h,+}) } d\xi + c^2 \right) 
\end{split} 
\end{equation*}

\noindent is invertible and $\boldsymbol{\Theta_h^{-1}}$ is analytic.
\end{lem}

\begin{proof}
This comes easily from the fact that 
\begin{equation*}
\begin{split}
\int^{\alpha^{\boldsymbol\epsilon}_{h,+}}_{\alpha^{\boldsymbol\epsilon}_{h,-}} \frac{\sqrt{(y-\alpha^{\boldsymbol\epsilon}_{h,-})(\alpha^{\boldsymbol\epsilon}_{h,+}-y)}}{y-\alpha^{\boldsymbol\epsilon}_{h,-}} dy = \pi \  \frac{\alpha^{\boldsymbol\epsilon}_{h,+}-\alpha^{\boldsymbol\epsilon}_{h,-}}{2} \\
\int^{\alpha^{\boldsymbol\epsilon}_{h,+}}_{\alpha^{\boldsymbol\epsilon}_{h,-}} \frac{\sqrt{(y-\alpha^{\boldsymbol\epsilon}_{h,-})(\alpha^{\boldsymbol\epsilon}_{h,+}-y)}}{y-\alpha^{\boldsymbol\epsilon}_{h,+}} dy = \pi \  \frac{\alpha^{\boldsymbol\epsilon}_{h,-}-\alpha^{\boldsymbol\epsilon}_{h,+}}{2} 
\end{split}
\end{equation*}
\end{proof}

\begin{proof}[Proof of Lemma \ref{thm:inv}]
By the identity (\ref{equilibre}) with $f(x) = (z-x)^{-1}$ and $z$ outside the support, we obtain that the Stieltjes transform $G(z) = \int \frac{1}{z-y} d\mu_{V}^{\boldsymbol\epsilon}(y)$ satisfies

\begin{equation*}
\frac{\beta}{2}G(z)^2 + G(z) \int \partial_1 T_t(z,y)d\mu_{V}^{\boldsymbol\epsilon}(y)  + F(z)=0  \ \ \ with \ \ F(z) = \iint \frac{ \partial_1 T_t(\tilde{y},y)- \partial_1 T_t(z,y)}{\tilde{y} - z} d\mu_{V}^{\boldsymbol\epsilon}(\tilde{y})d\mu_{V}^{\boldsymbol\epsilon}(y)
\end{equation*}

\noindent and this gives 
\begin{equation*}
\beta G(z) + \int  \partial_1T_t(z,y)d\mu_{V}^{\boldsymbol\epsilon}(y)  = - \sqrt{ \left(\int  \partial_1T_t(z,y)d\mu_{V}^{\boldsymbol\epsilon}(y)\right)^2 - 2 \beta F(z)} .
\end{equation*}

\noindent As $-\pi^{-1} \mathcal{I}G(z)$ converges towards the density of $\mu_{V}^{\boldsymbol\epsilon}$ as $z$ goes to the real axis (see for instance \cite{AGZ}, Section 2.4 for the  basic properties of the Stieltjes transform) and the quantity under the square root converges to a  real number, this number has to be negative on the support (otherwise the density would vanish) and thus for $x \in   A^{\boldsymbol\epsilon}$ 
\begin{equation*}
\frac{d\mu_{V}^{\boldsymbol\epsilon}}{dx} = \frac{1}{\beta \pi}\sqrt{2 \beta F(x) - \left(\int \partial_1T_t(x,y)d\mu_{V}^{\boldsymbol\epsilon}(y)\right)^2 } \ .
\end{equation*}

\noindent Noticing that $\sigma$ becomes purely imaginary when $z$ converges towards the support, we may write 
\begin{equation}\label{stieltjes}
 \beta G(z) + \int  \partial_1T_t(z,y)d\mu_{V}^{\boldsymbol\epsilon}(y) = \beta \pi \tilde{S}(z) \sigma(z)
\end{equation}

\noindent where  $\tilde{S}$ is an analytic extension of $S$  in $U$ (we can assume $\tilde{S}$ non zero  on $U$ by possibly shrinking $U$). We will keep writing $S$ for $\tilde{S}$.\\

\noindent For $f$ analytic in $U\setminus A^{\boldsymbol\epsilon}$ and $z \in U\setminus A^{\boldsymbol\epsilon}$ let

\begin{equation*}
\tilde{\Xi}f (z) =   \frac{i}{2} \oint \left( \frac{\beta f(\xi)}{z-\xi} - \partial_2T_t(z,\xi) f(\xi) \right){S}(\xi)\sigma(\xi) d\xi
\end{equation*} 

\noindent  where the contour  surrounds  $z$ and each $A^{\boldsymbol\epsilon}_h$. Then  $\tilde{\Xi}f \in \mathcal{O}(U\setminus A^{\boldsymbol\epsilon})$ and, noticing that $ - i S(x+i\delta) \sigma(x+i\delta) \underset{\delta \rightarrow 0^{+}}{\longrightarrow} \frac{d\mu_{V}^{\boldsymbol\epsilon}}{dx}$ , we have 

\begin{equation}
\begin{split}
\Xi f(z) &=  - \int \left(  \beta \frac{f(y)}{z-y} -\partial_2T_t(z,y)f(y) \right) d\mu_{V}^{\boldsymbol\epsilon}(y) + f(z) \left( \int \partial_1T_t(z,y) d\mu_{V}^{\boldsymbol\epsilon}(y)+\beta \int \frac{d\mu_{V}^{\boldsymbol\epsilon}(y)}{z-y} \right)\\
& = \frac{i}{2} \oint \left( \frac{\beta f(\xi)}{z-\xi} - \partial_2T_t(z,\xi) f(\xi) \right){S}(\xi)\sigma(\xi) d\xi + \beta \pi f(z) S(z) \sigma(z) \\
&= \tilde{\Xi}f (z)
\end{split}
\end{equation}

\noindent  where the contour surrounds each $A^{\boldsymbol\epsilon}_h$ (but not $z$), and we  used Cauchy's formula and (\ref{stieltjes}). If furthermore $f\in \mathcal{O}(U)$, by continuity this formula extends to $z \in U$.\\

\noindent Let $k  \in \mathcal{O}(U)$. We want to show that the function defined on each $U_h$ by \\

\begin{equation*}
f(z) = - \frac{1}{2 \beta \pi^2 \sigma(z) \sigma_h(z) S(z)} \left[ \oint \frac{i \sigma_h(\xi) (k(\xi)+c_h^1)}{(\xi - z) } d\xi + c_h^2\right]
\end{equation*}

\noindent where the contour surronds $A^{\boldsymbol\epsilon}_{h}$ and lays in $U_h$,   and $c_h^1 $ and $c_h^2$ are defined as in the statement of the lemma,  is a solution of $\Xi f = k + \kappa_k$ in $\mathcal{O}(U)$. The fact that $f \in \mathcal{O}(U)$ is clear (the function is meromorphic and the poles are removable by construction of $c^1$ and $c^2$). Thus, by previous remark, it suffices  to prove that $\tilde{\Xi} f = k + \kappa_k$.\\

\noindent By (\ref{potentiel}) We have 
\begin{equation}
\tilde{\Xi}f = (1-t) \ \tilde{\Xi}_0f + t \ \tilde{\Xi}_1 f +c_t
\end{equation}

\noindent where $c_t$ is a function constant on each $U_h$  depending on $t$ and 
\begin{equation*}
\left\lbrace
\begin{split}
\tilde{\Xi}_0f (z) &= \frac{\beta i}{2}\oint \frac{f(\xi)\sigma(\xi)S(\xi)}{z-\xi} d\xi  \\
\tilde{\Xi}_1f (z) &= \frac{\beta i}{2}\oint \frac{f(\xi)\sigma(\xi)S(\xi)}{z-\xi} d\xi
\end{split}
\right.
\end{equation*}

\noindent where the first contour surronds  $z$ and each $A^{\boldsymbol\epsilon}_{h'}$ whereas the second one surrounds $z$ and $A^{\boldsymbol\epsilon}_{h}$ when $z \in U_h$.\\

\noindent Let $f_0$ and $f_1$ be the functions analytic in $U\setminus A^{\boldsymbol\epsilon}$ defined on each $U_h\setminus A_h^{\boldsymbol\epsilon}$ by
\begin{equation*}
\begin{split}
f_0(z) &= - \frac{1}{2 \beta \pi^2 \sigma(z) \sigma_h(z) S(z)} \oint \frac{i \sigma_h(\xi) (k(\xi)+c_h^1)}{(\xi - z) } d\xi  \\
f_1(z) &= - \frac{c_h^2}{2 \beta \pi^2 \sigma(z) \sigma_h(z) S(z)}  
\end{split}
\end{equation*}

\noindent So that $f = f_0 + f_1$
\begin{equation*}
\begin{split}
 \tilde{\Xi}_0   (f_0)(z) &= \frac{\beta i}{2} \oint_{C} \frac{f_0(\xi){S}(\xi)\sigma(\xi)}{z-\xi} d\xi \\
       &=- \frac{\beta i}{2} \sum_h \oint_{C_h} \frac{1}{z-\xi}{\frac{1}{2 \beta \pi^2 \sigma(\xi)\sigma_h(\xi){S}(\xi)} \left( \oint_{C'_h} \frac{i \sigma_h(\eta) (k(\eta)+c_h^1)}{(\eta - \xi) } d\eta \right) {S}(\xi)\sigma(\xi)} d\xi \\
       &= \frac{1}{4 \pi^2} \sum_h \oint_{C_h} \oint_{C'_h} \frac{\sigma_h(\eta)(k(\eta)+c_h^1)}{(z-\xi)(\eta-\xi)\sigma_h(\xi)} d\eta d\xi 
\end{split}
\end{equation*}

\noindent where $C_h$ surrounds  $z$ and $A^{\boldsymbol\epsilon}_{h}$ (integral in $\xi$)  ,    and $C'_h$ surrounds $C_h$ (integral in $\eta$).\\

\noindent Cauchy formula gives 

$$
\oint_{C'_h} \frac{\sigma_h(\eta)(k(\eta)+c_h^1)}{(\eta-\xi)} d\eta = 2 i \pi (k(\xi)+c_h^1) \sigma_h(\xi) + \oint_{C''_h} \frac{\sigma_h(\eta)(k(\eta)+c_h^1)}{(\eta-\xi)} d\eta
$$

\noindent with $C_h$ surrounding $ C''_h$. Thus: 
\begin{equation*}
 \tilde{\Xi}_0 (f_0)(z) = \frac{1}{4 \pi^2} \sum_h \oint_{C_h} \oint_{C''_h} \frac{\sigma_h(\eta)(k(\eta)+c_h^1)}{(z-\xi)(\eta-\xi)\sigma_h(\xi)} d\eta d\xi + \frac{1}{4 \pi^2} \sum_h \oint_{C_h} \frac{2 i \pi (k(\xi)+c_h^1)}{z-\xi} d\xi.
\end{equation*}

\noindent Letting each $C_h$ go to infinity, we see that the first integral goes to zero and using  Cauchy formula again we see that the second term equals $k(z)+ c^1$.\\

\noindent We now prove $\tilde {\Xi}_0   (f_1)=0$ .
$$
\tilde {\Xi}_0   (f_1)(z)= -\frac{i}{4 \pi^2} \sum_h \oint_{C_h} \frac{c^2_h S(\xi)\sigma(\xi)}{\sigma(\xi)\sigma_h(\xi) S(\xi)(z-\xi)} d\xi = -\frac{i}{4 \pi^2} \sum_h \oint_{C_h} \frac{c^2_h}{\sigma_h(\xi)(z-\xi)} d\xi=0
$$

\noindent where we let the contours go to infinity.\\

\noindent By the exact same reasoning, we show that $\tilde{\Xi}_1 (f_0) = k + c^1 $ and $\tilde{\Xi}_1 (f_1) = 0$.\\

\noindent By setting $\kappa_k = c_t + c_h^1$ on each $U_h$ we have the desired result. The unicity of $\kappa_k$ is implied by the previous lemma. Formula (\ref{norme}) can be easily deduced by (\ref{formuleinverse}).

\end{proof}

\begin{rem} By Lemma \ref{surj} and (\ref{formuleinverse})   , if $k$   defined on $U\times U$ is analytic in each variable then $f$ defined on $U \times U$ and solution of

\begin{equation*}
\Xi f(\cdot , y) = k(x,y) +\kappa_k(x,y) \ \forall  y \in U \ ,
\end{equation*}

\noindent with $\kappa(.,y)$ constant on each $U_h$ is analytic in each variable.

\end{rem}

\noindent We can  now construct our approximate solution of the Monge-Ampère equation. As we want the domain $\mathbf{B}$ to be fixed by the flow of this approximate solution, we would like to choose   $\mathbf{y}_{1,t}^{\boldsymbol\epsilon}$ and $ \mathbf{z}_{t}^{\boldsymbol\epsilon}$ vanishing at the boundaries of $B$ (and $B\times B$). Fix $\delta >0$ small and denote $B^\delta =  \underset {{0\leq h \leq g}}{\bigcup} [\beta_{h,-} + \delta ; \beta_{h,+} - \delta]$. 

\noindent For a function $f:B \longrightarrow \mathbb{R}$ let $\Upsilon(f)$ be the multiplication of $f$ by a smooth plateau function equal to $1$ on $B^\delta$ and 0 outside $B$. If we are given a function $k \in \mathcal{O}(U)$ and $f \in \mathcal{O}(U) $ satisfying $\Xi(f) = k +\kappa_k$ , then :

\begin{itemize}
\item  $\Upsilon(f) = f  $ on $B^\delta$.
\item $\Upsilon(f)$  is  $C^{\infty}$ and has  compact support in $B$ (and can thus be extended by 0 to $\mathbb{R}$).
\item $\Xi(\Upsilon(f))= k + \kappa_k$  on $B^\delta$ (By definition of $\Xi$ and the fact that $f$ and $\Upsilon(f)$ coincide on $B^\delta$).
\item $\norm{\Upsilon(f)}_{C^j(\mathbb{R})} \leq C_j \norm{k}_{C^{j+2}(B)}$ for  some constants $C_j$.
\end{itemize}

\noindent Note that  by Remark \ref{controle}, possibly by shrinking $ \tilde{\mathcal{E}}$ we can assume  $\tilde{T_t}^{\boldsymbol\epsilon} < 0 $ outside $B^\delta$. Thus  for $N$ large enough and a constant $\eta >0$

\begin{equation}\label{LDP estimate}
\mathbb{P}_{T_t,B}^{N,\boldsymbol\epsilon} ( \exists i \  \lambda_i  \notin B^\delta ) \leq \exp(- N \eta) .
\end{equation}

\noindent  Moreover
\begin{equation}\label{ininv}
\int \left( \int \lvert k - \Xi(\Upsilon(f))\rvert dM_N \right) d \mathbb{P}_{T_t,B}^{N,\boldsymbol\epsilon} \leq \int \left(\int \lvert k - \Xi f\rvert dM_N \right) d \mathbb{P}_{T_t,B}^{N,\boldsymbol\epsilon}  + \int \left(\int \lvert \Xi f - \Xi(\Upsilon(f))\rvert dM_N \right)  d \mathbb{P}_{T_t,B}^{N,\boldsymbol\epsilon}  .
\end{equation}

\noindent The fist term on the right hand side is $0$ as $\kappa_k$ is constant on each $B_h$ and the second term is exponentially small by the large deviation estimate.\\

\noindent We first choose

\begin{equation*}
\left\{
\begin{split}
 & \mathbf{\hat{z}}^{\boldsymbol\epsilon}_{t}(\cdot , y)= \frac{1}{2} \ \Xi^{-1} 
 \left(  W(\cdot,y) \right) \ \forall  y\in B \\
  & \mathbf{\hat{y}}_{1,t}^{\boldsymbol\epsilon} = \left(\frac{\beta}{2}-1 \right) \Xi^{-1} \left(  \int \partial_1 \mathbf{\hat{z}}^{\boldsymbol\epsilon}_{t}(z,\cdot)d\mu_{V}^{\boldsymbol\epsilon}(z) \right)
\end{split}
\right.
\end{equation*}
\noindent and then

\begin{equation*}
\left\{
\begin{split}
 & \mathbf{z}^{\boldsymbol\epsilon}_{t}(\cdot , y) = \Upsilon(\mathbf{\hat{z}}^{\boldsymbol\epsilon}_{t}(\cdot , y))\ \forall  y\in B \\
 & \mathbf{y}_{1,t}^{\boldsymbol\epsilon} =  \Upsilon(\mathbf{\hat{y}}_{1,t}^{\boldsymbol\epsilon})
\end{split}
\right.
\end{equation*}

\noindent With this choice of function and by inequality (\ref{ininv}) we have that 
\begin{equation}\label{R}
\mathcal{R}= E + C_t^{N,\boldsymbol\epsilon}+ o\left(\frac{1}{N}\right) .
\end{equation}

\noindent We now have to control the error term $E$. To do so we will use  a direct consequence of  the concentration result proved in Corollary 3.5 of \cite{BGK} (adapted from a result from \cite{MS}):
\begin{prop}\label{concentration}
Let $V$ satisfy Hypothesis \ref{hypo} and  $T_t$ is as in (\ref{potentiel}).Then there exist  constants $c$ , $c'$ and $s_0$ such that for  $N$ large  enough , $s \geq s_0 \sqrt{\frac{\log N}{N}}$, and for any $\boldsymbol\epsilon = \mathbf{N}/N \in \tilde{\mathcal{E}}$, $t\in [0;1]$  we have  
\begin{equation}
\mathbb{P}_{T_t,B}^{N,\boldsymbol\epsilon} \left( \sup_{\substack{\phi\in C_c^1(B) \\ \norm{\phi'}_{\infty} \leq 1}} \left| \int \phi(x) d(L_N - \mu_{V}^{\boldsymbol\epsilon})(x) \right| \geq   s   \right)\leq \exp( - c N^2 s^2) + \exp( - c' N^2 ) .
\end{equation} 
\end{prop}

\noindent In order to control the error term we will make use of the following  three loop equations.  We recall that $M_N = N (L_N - N \mu_{V}^{\boldsymbol\epsilon})$ and we will denote $\tilde{M}_N = N (L_N - \mathbb{E}_{T_t,B}^{N,\boldsymbol\epsilon} [L_N])$.

\begin{lem}
Let $f \in C^1(B)$ such that for all $ 0 \leq h \leq g$  ,  $ f(\beta_{h,-}) = f( \beta_{h,+})=0$. Then
\begin{equation}\label{loop1}
\mathbb{E}_{T_t,B}^{N,\boldsymbol\epsilon} \left(  M_N(\Xi f) + \left(1- \frac{\beta}{2}\right) L_N(f') +  \frac{1}{N}  \left[  \iint \left( \frac{\beta}{2}  \frac{f(x)-f(y)}{x-y}+\partial_1 T_t(x,y) f(x) \right )dM_N(x) dM_N(y)   \right] \right) = 0 .
\end{equation}

\noindent If  $k_1$ is also in $ C^1(B)$ then 
\begin{equation}\label{loop2}
\begin{split}
\mathbb{E}_{T_t,B}^{N,\boldsymbol\epsilon} & \bigg( L_N(f k'_1) +  M_N(\Xi f)\tilde{M}_N(k_1) + \left(1- \frac{\beta}{2}\right) L_N(f')\tilde{M}_N(k_1)  \\
+&  \frac{1}{N}  \left[  \iint \left( \frac{\beta}{2}  \frac{f(x)-f(y)}{x-y}+\partial_1 T_t(x,y) f(x) \right )dM_N(x) dM_N(y)   \right] \tilde{M}_N(k_1) \bigg) = 0 .
\end{split}
\end{equation}

\noindent If $k_2$ and $k_3$ are also  in $ C^1(B)$ then
\begin{equation}\label{loop3}
\begin{split}
\mathbb{E}_{T_t,B}^{N,\boldsymbol\epsilon} & \bigg( \sum_{\sigma} L_N(fk'_{\sigma(1)})\tilde{M}_N(k_{\sigma(2)})\tilde{M}_N(k_{\sigma(3)}) +  M_N(\Xi f)\tilde{M}_N(k_1)\tilde{M}_N(k_2)\tilde{M}_N(k_3)   \\
+&  \frac{1}{N}  \left[   \iint \left( \frac{\beta}{2}  \frac{f(x)-f(y)}{x-y}+\partial_1 T_t(x,y) f(x) \right )dM_N(x) dM_N(y)   \right] \tilde{M}_N(k_1)\tilde{M}_N(k_2)\tilde{M}_N(k_3)\\
+& \left(1- \frac{\beta}{2}\right) L_N(f')\tilde{M}_N(k_1)\tilde{M}_N(k_2)\tilde{M}_N(k_3)  \bigg) = 0 .
\end{split}
\end{equation}

\noindent where the sum ranges over the permutations of $\mathfrak{S}_3$
\end{lem}
\begin{proof}
\noindent Using integration by parts we show 
\begin{equation}\label{loop}
\mathbb{E}_{T_t,B}^{N,\boldsymbol\epsilon} \left( \iint \left( \frac{\beta}{2}    \frac{f(x)-f(y)}{x-y}+\partial_1 T_t(x,y) f(x) \right )dL_N(x) dL_N(y)    +\frac{1}{N} \left(1- \frac{\beta}{2}\right) L_N(f')    \right) = 0
\end{equation}

\noindent we deduce the first loop equation by using the definition of $\Xi$.\\

\noindent The second loop equation is obtained by replacing in (\ref{loop}) $T_t(x,y)$ by $T_t(x,y) - \delta_1 (k_1(x) + k_1(y))$ and differentiating at $\delta=0$ .\\

\noindent  The third one is obtained by replacing  in (\ref{loop}) $T_t(x,y)$ by $T_t(x,y) - \delta_1 (k_1(x) + k_1(y)) - \delta_2 (k_2(x) + k_2(y)) - \delta_3 (k_3(x) + k_3(y)) $ and differentiating at $\delta_1=\delta_2=\delta_3=0$ .

\end{proof}

\noindent We will now put in use these loop equations and the concentration result of Proposition \ref{concentration} to obtain some estimates.

\begin{lem}\label{mnexp}
Let $k$ be an analytic function on $U$. Then for some constant $C$:
\begin{equation*}
\begin{split}
\left| \right. \mathbb{E}_{T_t,B}^{N,\boldsymbol\epsilon} & \left. \left( M_N(k) \right) \right| \leq C   \log N \ \norm{k}_{C^{6}(B)} . \\
 \mathbb{E}_{T_t,B}^{N,\boldsymbol\epsilon} & \left( {M}_N(k)^2 \right) \ \leq C   (\log N)^2 \ \norm{k}_{C^{6}(B)}^2 . \\
\mathbb{E}_{T_t,B}^{N,\boldsymbol\epsilon} & \left( {M}_N(k)^4 \right)  \leq C   (\log N)^4 \ \norm{k}_{C^{6}(B)}^4 .
\end{split}
\end{equation*}
\end{lem}

\begin{proof}
We apply (\ref{loop1}) to $f = \Upsilon(\Xi^{-1} k)$. Using (\ref{ininv}) we obtain

\begin{equation*}
\begin{split}
\mathbb{E}_{T_t,B}^{N,\boldsymbol\epsilon}  \bigg(  M_N(k)  + & \left. \frac{1}{N}  \left[\frac{\beta}{2}   \iint \left( \frac{\Upsilon(\Xi^{-1} k)(x)-\Upsilon(\Xi^{-1} k)(y)}{x-y}+\partial_1 T_t(x,y) \Upsilon(\Xi^{-1} k)(x) \right )dM_N(x) dM_N(y)   \right]  \right. \\
 + &    \left(1- \frac{\beta}{2}\right) L_N((\Upsilon(\Xi^{-1} k))')  \bigg) = O \left(N \norm{ k}_{L^{\infty}(B)} \exp(- N \eta) \right) .
\end{split}
\end{equation*}

\noindent Let 
\begin{equation*}
\begin{split}
\Lambda(k) = \frac{1}{N}  \Bigg[\frac{\beta}{2}   \iint \Big( \frac{\Upsilon(\Xi^{-1} k)(x)-\Upsilon(\Xi^{-1} k)(y)}{x-y}    + &  \partial_1 T_t(x,y) \Upsilon(\Xi^{-1} k)(x) \Big) dM_N(x) dM_N(y)   \Bigg]   \\
 + &   \left(1- \frac{\beta}{2}\right) L_N((\Upsilon(\Xi^{-1} k))') .
\end{split}
\end{equation*}
\noindent Denoting by $\mathcal{F}$ the fourier transform operator (for functions of either one or several variables) we have 
\begin{equation*}
\begin{split}
\iint  \frac{\Upsilon(\Xi^{-1} k)(x)-\Upsilon(\Xi^{-1} k)(y)}{x-y} &  dM_N(x) dM_N(y)\\
& = i \int \left( \int_0^1 d\alpha \int e^{i \alpha \xi x} dM_N(x) \int e^{i (1-\alpha) \xi y} dM_N(y) \right) \mathcal{F}({\Upsilon(\Xi^{-1} k)})(\xi) \xi d\xi
\end{split}
\end{equation*}
\noindent and  
\begin{equation*}
\begin{split}
\iint  \partial_1 T_t(x,y) \Upsilon(\Xi^{-1} k)(x) &  dM_N(x) dM_N(y)\\
& =  \int \left(  \int e^{i  \xi x} dM_N(x) \int e^{i \zeta y} dM_N(y) \right) \mathcal{F}(\partial_1 T_t \  \Upsilon(\Xi^{-1} k))(\xi, \zeta)  d\xi d\zeta .
\end{split}
\end{equation*}

\noindent Now on the set $\Omega = \left\lbrace \sup_{\substack{\phi\in C_c^1(B) \\ \norm{\phi'}_{\infty} \leq 1}} \left| \int \phi(x) d(L_N - \mu_{V}^{\boldsymbol\epsilon})(x) \right| \leq s_0 \sqrt{\frac{\log N}{N}} \right\rbrace$ we have 
\begin{equation*}
\begin{split}
\left| \int e^{i  \xi x} dM_N(x)   \right| & \leq \left| \int \Upsilon(e^{i \xi \cdot })(x) dM_N(x)   \right| + 2 N e^{-N\eta} \\
& \leq C (1 + |\xi|) \sqrt{N \log N} + 2 N e^{-N\eta}
\end{split}
\end{equation*}
\noindent consequently, on this set

\begin{equation*}
\begin{split}
\Bigg| \iint  \frac{\Upsilon(\Xi^{-1} k)(x)-\Upsilon(\Xi^{-1} k)(y)}{x-y} &  dM_N(x) dM_N(y)\Bigg| \\
& \leq C (N \log N) \int  |\mathcal{F}({\Upsilon(\Xi^{-1} k)})(\xi)| (1 + |\xi|)^3 d\xi + O(N e^{-N\eta}) .
\end{split}
\end{equation*}

\noindent The integral is bounded by the norm $\mathcal{H}^4(\mathbb{R})$ of $\Upsilon(\Xi^{-1} k)$ and we have:

\begin{equation*}
\norm{\Upsilon(\Xi^{-1} k)}_{\mathcal{H}^4(\mathbb{R})} \leq  C  \Big(
\norm{\Upsilon(\Xi^{-1} k)}_{\mathcal{L}^2(\mathbb{R})} + \norm{\big( \Upsilon(\Xi^{-1} k)\big)^{(4)}}_{\mathcal{L}^2(\mathbb{R})} \Big) .
\end{equation*}

\noindent As ${\Upsilon(\Xi^{-1} k)}$  has its support in $B$, the $\mathcal{L}^2(\mathbb{R})$ norm can be in turn controlled by the $\mathcal{L}^{\infty}(\mathbb{R})$ norm and we can use (\ref{norme}). Similarly on $\Omega$  we have 
\begin{equation*}
\Bigg| \iint \partial_1 T_t(x,y) \Upsilon(\Xi^{-1} k)(x)  dM_N(x) dM_N(y)\Bigg| \leq C(N \log N) \norm{k}_{C^6(B)} + O(N e^{-N \eta})
\end{equation*}

\noindent Note that here the constant depends on $T_t$ but we can make it uniform in $t$ and $\boldsymbol\epsilon \in \tilde{ \mathcal{E}}$.

\noindent On $\Omega^c$ we can use the trivial bound 
\begin{equation*}
\left| \int e^{i  \xi x} dM_N(x)   \right| \leq 2 N
\end{equation*}

\noindent to prove that $|\Lambda(k)|$ is bounded everywhere by $C N    \norm{k}_{C^{6}(B)}$. By  using Proposition \ref{concentration} we obtain

\begin{equation*}
| \mathbb{E}_{T_t,B}^{N,\boldsymbol\epsilon}  (\Lambda (k)) | \leq  C \Big( (\log N) \norm{k}_{C^6(B)} +N e^{-c s_0^2 N \log N} \norm{k}_{C^{6}(B)} \Big)
\end{equation*}

\noindent and we can conclude the proof of the first inequality.

\noindent To prove the second inequality, using (\ref{loop2}) and (\ref{ininv}) we have 

\begin{equation*}
\mathbb{E}_{T_t,B}^{N,\boldsymbol\epsilon} \Big(M_N(k) \tilde{M}_N(k)\Big) =  - \mathbb{E}_{T_t,B}^{N,\boldsymbol\epsilon} \Big(
\Lambda(k)\tilde{M}_N(k) + L_N ( k' \  \Upsilon(\Xi^{-1} k)) \Big) + O (N^2 \norm{ k}_{L^{\infty}(B)}^2 e^{-N \eta}) .
\end{equation*}

\noindent By splitting on $\Omega$ and $\Omega^c$ we see that 
\begin{equation*}
\begin{split}
\Big| \mathbb{E}_{T_t,B}^{N,\boldsymbol\epsilon} \Big(M_N(k) \tilde{M}_N(k)\Big)   \Big| & \leq C \Big( \log N \norm{ k}_{C^{6}(B)} \mathbb{E}_{T_t,B}^{N,\boldsymbol\epsilon} \Big(|\tilde{M}_N(k)|\Big) +  \norm{k}_{C^{6}(B)}^2 \Big(1+N^2 e^{-c s_0^2 N \log N}  + N^2  e^{-N \eta}\Big) \Big) 
\end{split}
\end{equation*}

\noindent We notice that $M_N(k) - \tilde{M}_N(k)  = \mathbb{E}_{T_t,B}^{N,\boldsymbol\epsilon} ( M_N(k)) $ is deterministic and that $\mathbb{E}_{T_t,B}^{N,\boldsymbol\epsilon} ( \tilde{M}_N(k))$  vanishes. The term on the left is thus equal to $\Big| \mathbb{E}_{T_t,B}^{N,\boldsymbol\epsilon} \Big(\tilde{M}_N(k)^2\Big)   \Big|$ and we obtain 

\begin{equation*}
 \mathbb{E}_{T_t,B}^{N,\boldsymbol\epsilon} \Big( \tilde{M}_N(k)^2\Big)    \leq C \Big( \log N \norm{ k}_{C^{6}(B)} \sqrt{\mathbb{E}_{T_t,B}^{N,\boldsymbol\epsilon} \Big(\tilde{M}_N(k)^2\Big)} +  \norm{k}_{C^{6}(B)}^2 \Big(1+N^2 e^{-c s_0^2 N \log N}  + N^2  e^{-N \eta}\Big) \Big) .
\end{equation*}

\noindent Elementary manipulations show that this implies that $\mathbb{E}_{T_t,B}^{N,\boldsymbol\epsilon} \Big( \tilde{M}_N(k)^2\Big)   \leq C (\log N)^2 \norm{ k}_{C^{6}(B)}^2$ with a different constant.\\

\noindent  Writing
\begin{equation*}
\begin{split}
\mathbb{E}_{T_t,B}^{N,\boldsymbol\epsilon} \Big( {M}_N(k)^2\Big) & \leq 2 \ \mathbb{E}_{T_t,B}^{N,\boldsymbol\epsilon} \Big( \tilde{M}_N(k)^2\Big) + 2 \ \mathbb{E}_{T_t,B}^{N,\boldsymbol\epsilon} \Big( (\tilde{M}_N(k) - M_N(k))^2\Big) \\
& =  2 \  \mathbb{E}_{T_t,B}^{N,\boldsymbol\epsilon} \Big( \tilde{M}_N(k)^2\Big) + 2 \ \mathbb{E}_{T_t,B}^{N,\boldsymbol\epsilon} ( M_N(k))^2
\end{split}
\end{equation*}

\noindent and using the first inequality yields to the second one.\\

\noindent Finally, to prove the last inequality, (\ref{loop3}) gives :
\begin{equation*}
\mathbb{E}_{T_t,B}^{N,\boldsymbol\epsilon} \Big( \tilde{M}_N(k)^4\Big)    \leq C  \Big( \log N \norm{ k}_{C^{6}(B)} {\mathbb{E}_{T_t,B}^{N,\boldsymbol\epsilon} \Big(\tilde{M}_N(k)^4\Big)}^{\frac{3}{4}} +  \norm{k}_{C^{6}(B)}^4 \Big((\log N)^2+N^4 e^{-c s_0^2 N \log N}  + N^4  e^{-N \eta}\Big) \Big) 
\end{equation*}

\noindent which shows $\mathbb{E}_{T_t,B}^{N,\boldsymbol\epsilon} \Big( \tilde{M}_N(k)^4\Big)   \leq C (\log N)^4 \norm{ k}_{C^{6}(B)}^4$. We conclude by using the identity 

\begin{equation*}
\mathbb{E}_{T_t,B}^{N,\boldsymbol\epsilon} \Big( {M}_N(k)^4\Big)  \leq 8 \ \mathbb{E}_{T_t,B}^{N,\boldsymbol\epsilon} \Big( \tilde{M}_N(k)^4\Big) + 8 \ \mathbb{E}_{T_t,B}^{N,\boldsymbol\epsilon} \Big( (\tilde{M}_N(k) - M_N(k))^4 \Big) .
\end{equation*}

\end{proof}
\noindent We will need a last lemma to estimate the error $E$.
\begin{lem}\label{controlefourier}
\noindent  There exists  a constant $C$ such that for   $\phi \in C^{\infty}(\mathbb{R})$ (resp. $\psi \in C^{\infty}(\mathbb{R}^2)$, $\chi \in C^{\infty}(\mathbb{R}^3)$)  of compact support in $B$ (resp. $B^2$, $B^3$)   we have 
\begin{equation*}
\begin{split}
 \mathbb{E}_{T_t,B}^{N,\boldsymbol\epsilon} \bigg(&\bigg| \int  \phi(x) dM_N(x)  \bigg| \bigg)  \leq C \norm{\phi}_{C^{6}(B)} \log N \\
 \mathbb{E}_{T_t,B}^{N,\boldsymbol\epsilon} \bigg(\bigg| \iint  \frac{\phi(x) - \phi(y)}{x-y} &dM_N(x) dM_N(y) \bigg| \bigg)    \leq C \norm{\phi}_{C^{14}(B)} \log N^2 \\
 \mathbb{E}_{T_t,B}^{N,\boldsymbol\epsilon} \bigg( \bigg| \iint \psi(x,y) &dM_N(x) dM_N(y) \bigg| \bigg)   \leq C \norm{\psi}_{C^{14}(B^2)} \log N^2 \\
\mathbb{E}_{T_t,B}^{N,\boldsymbol\epsilon} \bigg( \bigg| \iiint \chi(x,y,z) &dM_N(x) dM_N(y) dM_N(z)  \bigg| \bigg)   \leq C \norm{\chi}_{C^{21}(B^3)} \log N^3 \\
 \mathbb{E}_{T_t,B}^{N,\boldsymbol\epsilon} \bigg( \bigg| \iiint \frac{\psi(x,y) - \psi(z,y)}{x-z} &dM_N(x) dM_N(y) dM_N(z) \bigg| \bigg)   \leq C \norm{\psi}_{C^{21}(B^2)} \log N^3 
\end{split}
\end{equation*}

\end{lem}
\begin{proof} 
\noindent We will prove the last inequality as the other ones are simpler and can  be proved the  same way.

\begin{equation*}
\begin{split}
\iiint \frac{\psi(x,y) - \psi(z,y)}{x-z}& dM_N(x)  dM_N(y) dM_N(z)  \\&  = i \iint  \Big(\int_0^1 d\alpha M_N (e^{i \alpha \xi \cdot})  M_N(e^{i (1- \alpha) \xi \cdot})   M_N(e^{i  \zeta \cdot})\Big) \mathcal{F}(\partial_1\psi)(\xi,\zeta) d\xi d\zeta
\end{split} 
\end{equation*}

\noindent and by using Hölder inequality we obtain
\begin{equation*}
\begin{split}
&  \mathbb{E}_{T_t,B}^{N,\boldsymbol\epsilon} \bigg( \bigg|  \iiint  \frac{\psi(x,y) - \psi(z,y)}{x-z} dM_N(x) dM_N(y) dM_N(z) \bigg| \bigg)  \\
 & \leq  \iint \Big( \mathbb{E}_{T_t,B}^{N,\boldsymbol\epsilon} \Big( {M}_N(e^{i \alpha \xi \cdot})^4\Big)^{\frac{1}{4}} \mathbb{E}_{T_t,B}^{N,\boldsymbol\epsilon} \Big( {M}_N(e^{i (1- \alpha) \xi \cdot})^4\Big)^{\frac{1}{4}} \mathbb{E}_{T_t,B}^{N,\boldsymbol\epsilon} \Big( {M}_N(e^{i \zeta \cdot})^4\Big)^{\frac{1}{4}} \Big) |\xi| \mathcal{F}(\psi)(\xi,\zeta) d\xi d\zeta \\
 & \leq C (\log N)^3 \iint (1 + |\xi|^6)^2 (1 + |\zeta|^6) |\xi| | \mathcal{F}(\psi)(\xi,\zeta)| d\xi d\zeta
\end{split}
\end{equation*}

\noindent where we used the last identity of Lemma \ref{mnexp}. The last term is controled by the $H^{21}(\mathbb{R}^2)$ norm of $\psi$ and we have  
\begin{equation*}
\norm{\psi}_{\mathcal{H}^{21}(\mathbb{R}^2)} \leq  C  \Big(
\norm{\psi}_{\mathcal{L}^2(\mathbb{R}^2)} + \sup_{|\beta|\leq 21}\norm{\partial^{\beta}\psi}_{\mathcal{L}^2(\mathbb{R}^2)} \Big) \leq C \norm{\psi}_{C^{21}(B^2)} .
\end{equation*}
\end{proof}

\noindent A direct application of this lemma shows that $\mathbb{E}_{T_t,B}^{N,\boldsymbol\epsilon} (\big|E \big|) \leq C \frac{(\log N)^3}{N}$, and we could prove similarly using higher order loop equations that for all integer $k\geq 1$ 

\begin{equation}\label{controlepE}
\Big( \mathbb{E}_{T_t,B}^{N,\boldsymbol\epsilon} (\big|E \big|^ {2k}) \Big)^{1/2k}\leq C_k \frac{(\log N)^3}{N} .
\end{equation}

\noindent In order to prove Propostion \ref{erreur} it remains to control the deterministic term $C_t^{N,\boldsymbol\epsilon}$. Let 
\begin{equation*}
\begin{split}
\boldsymbol{\mathcal{L}}(\mathbf{Y})= &  \beta \sum_{h=0}^g  \sum_{1\leq i < j \leq N_h}  \frac{\mathbf{Y}_{h,i} - \mathbf{Y}_{h,j}}{\lambda_{h,i} - \lambda_{h,j}} + \beta \sum_{0\leq h < h' \leq g} \sum_{\substack{1\leq i \leq N_h \\ 1\leq j \leq N_{h'}}} \frac{\mathbf{Y}_{h,i} - \mathbf{Y}_{h',j}}{\lambda_{h,i} - \lambda_{h',j}}  \\
& + {\sum_{0\leq h, h' \leq g} \sum_{\substack{1\leq i \leq N_h \\ 1\leq j \leq N_{h'}}} } (\partial_1 T(\lambda_{h,i}, \lambda_{h',j}) \mathbf{Y}_{h,i,t} )  + \divergence { (\mathbf{Y})} .
\end{split}
\end{equation*}

\noindent Integration by part shows that  any vector field $\mathbf{Y}$ that vanishes on the boundary of $\textbf{B}$ satisfies \   $\mathbb{E}_{T_t,B}^{N,\boldsymbol\epsilon} (\boldsymbol{\mathcal{L}}(\mathbf{Y}))=0$.
\noindent Thus
\begin{equation*}
\begin{split}
\mathbb{E}_{T_t,B}^{N,\boldsymbol\epsilon} \big(\mathcal{R}_t^{N,\boldsymbol\epsilon}(\mathbf{Y}^{N,\boldsymbol\epsilon}) \big) & = \mathbb{E}_{T_t,B}^{N,\boldsymbol\epsilon} \Big( \ - \frac{1}{2} {\sum_{0\leq h, h' \leq g} \sum_{\substack{1\leq i \leq N_h \\ 1\leq j \leq N_{h'}}} } W(\lambda_{h,i}, \lambda_{h',j}) + N  \sum_{0\leq h \leq g} \sum_{{1\leq i \leq N_h }} \int W(\lambda_{h,i},z)d\mu_{V}^{\boldsymbol\epsilon}(z)    \\
& +   \boldsymbol{\mathcal{L}} (\mathbf{Y}_t^{N,\boldsymbol\epsilon}) - c_t^{N,\boldsymbol\epsilon} \Big) = 0 \ , 
\end{split}
\end{equation*}

\noindent and by (\ref{R})
\begin{equation*}
| C_t^{N,\boldsymbol\epsilon}|  = \Big|\mathbb{E}_{T_t,B}^{N,\boldsymbol\epsilon} \big(\mathcal{R}_t^{N,\boldsymbol\epsilon}(\mathbf{Y}^{N,\boldsymbol\epsilon}) - E \big)  \Big| +o\Big(\frac{1}{N}\Big)\leq C \frac{(\log N)^3}{N} \ .
\end{equation*}

\subsection{Obtaining the Transport map via the flow}
\noindent In this section we will discuss the properties of the transport map given by the flow of the approximate solution $\mathbf{Y}^{N,\boldsymbol\epsilon}$ of the Monge-Ampère equation. As the equilibrium measures of the initial potential and the target potential are the same, this map is equal to the identity at the first order. The smaller order are then given by the expansion (\ref{expansion}) of $\mathbf{Y}^{N,\boldsymbol\epsilon}$ .

\begin{lem}
Let $V$ satisfy Hypothesis \ref{hypo} ,  $T_t$ is as in (\ref{potentiel}) and $\boldsymbol\epsilon = \mathbf{N}/N \in \tilde{\mathcal{E}}$. Then the flow  ${X}_t^{N,\boldsymbol\epsilon}$ can be written

\begin{equation}\label{decomposition}
{X}_t^{N,\boldsymbol\epsilon} = Id + \frac{1}{N}{X}_t^{N,\boldsymbol\epsilon,1} + \frac{1}{N^2}{X}_t^{N,\boldsymbol\epsilon,2}
\end{equation}

\noindent where ${X}_t^{N,\boldsymbol\epsilon,1}$ and ${X}_t^{N,\boldsymbol\epsilon,2}$ are in $C^{\infty}(\mathbb{R}^N)$  supported in $\mathbf{B}$, and for some constant $C>0$

\begin{equation}\label{controle1}
\sup_{\substack{0\leq h \leq g \\ 1\leq i \leq N_{h}}} \norm{{X}_{h,i,t}^{N,\boldsymbol\epsilon,1}}_{L^4( \mathbb{P}_{V,B}^{N,\boldsymbol\epsilon})} \leq C \log N \ \ \ , \ \ \ \norm{{X}_{t}^{N,\boldsymbol\epsilon,2}}_{L^2( \mathbb{P}_{V,B}^{N,\boldsymbol\epsilon})} \leq C \sqrt{N}(\log N)^2 
\end{equation}

\noindent and with probability greater than $1 - N^{-\frac{N}{C}}$

\begin{equation}\label{controleespace}
\sup_{\substack{0\leq h \leq g \\ 1\leq i,j \leq N_{h}}} \big| {X}_{h,i,t}^{N,\boldsymbol\epsilon,1}(\boldsymbol{\lambda}) - {X}_{h,j,t}^{N,\boldsymbol\epsilon,1}(\boldsymbol{\lambda}) \big| \leq C \sqrt{N}\log N  |\lambda_{h,i} - \lambda_{h,j} |
\end{equation}

\begin{equation}\label{controleespace2}
\sup_{\substack{0\leq h \leq g \\ 1\leq i,j \leq N_{h}}} \big| {X}_{h,i,t}^{N,\boldsymbol\epsilon,2}(\boldsymbol{\lambda}) - {X}_{h,j,t}^{N,\boldsymbol\epsilon,2}(\boldsymbol{\lambda}) \big| \leq C N \sqrt{N}\log N  |\lambda_{h,i} - \lambda_{h,j} |
\end{equation}

\begin{equation}\label{controlenorme}
\sup_{\substack{0\leq h \leq g \\ 1\leq i \leq N_{h}}} \norm{ {X}_{h,i,t}^{N,\boldsymbol\epsilon,1} }_{\infty} \leq C \sqrt{N}\log N  \ \ \  , \ \ \  \sup_{\substack{0\leq h \leq g \\ 1\leq i \leq N_{h}}} \norm{ {X}_{h,i,t}^{N,\boldsymbol\epsilon,2} }_{\infty} \leq C N \sqrt{N}\log N  
\end{equation}

\end{lem}
\begin{proof}
\noindent The expansion (\ref{expansion}) suggests  to define $ {X}_{t}^{N,\boldsymbol\epsilon,1} = ( {X}_{0,1,t}^{N,\boldsymbol\epsilon,1}, \cdots , {X}_{g,N_g,t}^{N,\boldsymbol\epsilon,1}) $ as  the solution of the linear ODE
\begin{equation}\label{equation1}
\begin{split}
 \dot{X}_{h,i,t}^{N,\boldsymbol\epsilon,1}(\boldsymbol{\lambda}) = \mathbf{y}^{\boldsymbol\epsilon}_{1,t} (\lambda_{h,i}) + \int \mathbf{z}^{\boldsymbol\epsilon}_{t}(\lambda_{h,i},y) dM_N(y) + \frac{1}{N} \sum_{0 \leq h' \leq g} \sum_{1 \leq j \leq N_{h'}} \partial_2 \mathbf{z}^{\boldsymbol\epsilon}_{t}(\lambda_{h,i}, \lambda_{h',j}) {X}_{h',j,t}^{N,\boldsymbol\epsilon,1}(\boldsymbol{\lambda}) \
\end{split}
\end{equation}

\noindent with initial condition ${X}_{t}^{N,\boldsymbol\epsilon,1}=0$. We then define ${X}_{t}^{N,\boldsymbol\epsilon,2}$ through the identity (\ref{decomposition}).

\noindent Using the fact that $\mathbf{y}^{\boldsymbol\epsilon}_{1,t}$ and $\mathbf{z}^{\boldsymbol\epsilon}_{t}$ have compact support and are thus bounded, along with equation (\ref{equation1}), we obtain:
\begin{equation*}
\frac{d}{dt}\Big( \sup_{\substack{0\leq h \leq g \\ 1\leq i \leq N_{h}}} \norm{{X}_{h,i,t}^{N,\boldsymbol\epsilon,1}}_{L^4( \mathbb{P}_{V,B}^{N,\boldsymbol\epsilon})}\Big) \leq C \Big( 1 + \sup_{\substack{0\leq h \leq g \\ 1\leq i \leq N_{h}}} \norm{{X}_{h,i,t}^{N,\boldsymbol\epsilon,1}}_{L^4( \mathbb{P}_{V,B}^{N,\boldsymbol\epsilon})} + \sup_{\substack{0\leq h \leq g \\ 1\leq i \leq N_{h}}} \norm{\int \mathbf{z}^{\boldsymbol\epsilon}_{t}(\lambda_{h,i},y)dM_N(y)}_{L^4( \mathbb{P}_{V,B}^{N,\boldsymbol\epsilon})} \Big).
\end{equation*}
\noindent As in Lemma \ref{controlefourier}, we can prove that the last term is of order $\log N$. Using Grönwall's Lemma, this proves 

\begin{equation}\label{controle11}
\sup_{\substack{0\leq h \leq g \\ 1\leq i \leq N_{h}}} \norm{{X}_{h,i,t}^{N,\boldsymbol\epsilon,1}}_{L^4( \mathbb{P}_{V,B}^{N,\boldsymbol\epsilon})} \leq C \log N .
\end{equation}

\noindent Furthermore, Proposition \ref{concentration} shows that for some constant $C$, with probability greater than $1 - N^{-\frac{N}{C}}$ we have 
\begin{equation*}
\norm {\int \partial_1 \mathbf{z}^{\boldsymbol\epsilon}_{t}(  \cdot , y) dM_N(y) }_{\infty} \leq C \sqrt{N} \log N
\end{equation*}

\noindent and similarly, this proves (\ref{controleespace}). We now have to   bound the norm    of ${X}_{t}^{N,\boldsymbol\epsilon,2}$. For $s\in [0;1]$ let 
\begin{equation*}
{X}_{t}^{s,N,\boldsymbol\epsilon} = Id + \frac{s}{N}{X}_t^{N,\boldsymbol\epsilon,1} + \frac{s}{N^2}{X}_t^{N,\boldsymbol\epsilon,2} = (1-s) Id + s {X}_{t}^{N,\boldsymbol\epsilon}
\end{equation*}

\noindent and define the measure $M_N^{{X}_{t}^{s,N,\boldsymbol\epsilon}}$ by
\begin{equation*}
\int f(y) dM_N^{{X}_{t}^{s,N,\boldsymbol\epsilon}}(y) = \sum_{0 \leq h \leq g} \sum_{1 \leq i \leq N_h} f({X}_{h,i,t}^{s,N,\boldsymbol\epsilon}(\boldsymbol{\lambda})) - N\int f d\mu_{V}^{\boldsymbol\epsilon} .
\end{equation*}

\noindent Then a Taylor expansion gives us an ODE for ${X}_{t}^{N,\boldsymbol\epsilon,2}$

\begin{equation}
\begin{split}
\dot{X}_{h,i,t}^{N,\boldsymbol\epsilon,2} (\boldsymbol{\lambda}) & = \int_0^1 (\mathbf{y}^{\boldsymbol\epsilon }_{1,t})' \big( {X}_{h,i,t}^{s,N,\boldsymbol\epsilon} (\boldsymbol{\lambda}) \big) ds \ \ \Big( {X}_{h,i,t}^{N,\boldsymbol\epsilon,1}(\boldsymbol{\lambda}) + \frac{1}{N}{X}_{h,i,t}^{N,\boldsymbol\epsilon,2}(\boldsymbol{\lambda}) \Big) \\
&+ \int_0^1 \Big[ \int \partial_1\mathbf{z}^{\boldsymbol\epsilon }_{t} \big({X}_{h,i,t}^{s,N,\boldsymbol\epsilon}(\boldsymbol{\lambda}) ,y \big) dM_N^{{X}_{t}^{s,N,\boldsymbol\epsilon}} - \int \partial_1\mathbf{z}^{\boldsymbol\epsilon }_{t} \big(\lambda_{h,i},y \big) dM_N(y) \Big] ds \ \  \Big( {X}_{h,i,t}^{N,\boldsymbol\epsilon,1}(\boldsymbol{\lambda}) + \frac{1}{N}{X}_{h,i,t}^{N,\boldsymbol\epsilon,2}(\boldsymbol{\lambda}) \Big)\\
&+  \int \partial_1\mathbf{z}^{\boldsymbol\epsilon }_{t} \big(\lambda_{h,i} ,y \big) dM_N(y) \ \ \Big( {X}_{h,i,t}^{N,\boldsymbol\epsilon,1}(\boldsymbol{\lambda}) + \frac{1}{N}{X}_{h,i,t}^{N,\boldsymbol\epsilon,2}(\boldsymbol{\lambda}) \Big)\\
& + \sum_{0 \leq h' \leq g} \sum_{1 \leq j \leq N_{h'}} \int_0^1 \Big[ \partial_2\mathbf{z}^{\boldsymbol\epsilon }_{t} \big({X}_{h,i,t}^{s,N,\boldsymbol\epsilon}(\boldsymbol{\lambda}) ,{X}_{h',j,t}^{s,N,\boldsymbol\epsilon}(\boldsymbol{\lambda}) \big)- \partial_2\mathbf{z}^{\boldsymbol\epsilon }_{t} \big( \lambda_{h,i},\lambda_{h',j}\big) \Big]ds \ {X}_{h',j,t}^{N,\boldsymbol\epsilon,1}(\boldsymbol{\lambda}) \\
&+ \sum_{0 \leq h' \leq g} \sum_{1 \leq j \leq N_{h'}} \int_0^1 \Big[ \partial_2\mathbf{z}^{\boldsymbol\epsilon }_{t} \big({X}_{h,i,t}^{s,N,\boldsymbol\epsilon}(\boldsymbol{\lambda}) ,{X}_{h',j,t}^{s,N,\boldsymbol\epsilon}(\boldsymbol{\lambda}) \big) \Big]ds \  \ \frac{{X}_{h',j,t}^{N,\boldsymbol\epsilon,2}(\boldsymbol{\lambda})}{N} .
\end{split}
\end{equation}

\noindent We then use the bounds
\begin{equation*}
\begin{split}
\int_0^1 \Big| \int \partial_1\mathbf{z}^{\boldsymbol\epsilon }_{t} \big({X}_{h,i,t}^{s,N,\boldsymbol\epsilon}(\boldsymbol{\lambda}) ,y \big) dM_N^{{X}_{t}^{s,N,\boldsymbol\epsilon}} - &\int \partial_1\mathbf{z}^{\boldsymbol\epsilon }_{t} \big(\lambda_{h,i},y \big) dM_N(y) \Big|ds \\
& \leq C  | {X}_{h,i,t}^{N,\boldsymbol\epsilon,1}| +\frac{C}{N} |{X}_{h,i,t}^{N,\boldsymbol\epsilon,2}| +\frac{C}{N}\sum_{h',j}\Big(   | {X}_{h',j,t}^{N,\boldsymbol\epsilon,1}| +\frac{1}{N} |{X}_{h',j,t}^{N,\boldsymbol\epsilon,2}| \Big)  , \\
\sum_{h',j} \int_0^1 \Big| \partial_2\mathbf{z}^{\boldsymbol\epsilon }_{t} \big({X}_{h,i,t}^{s,N,\boldsymbol\epsilon}(\boldsymbol{\lambda}) ,{X}_{h',j,t}^{s,N,\boldsymbol\epsilon}(\boldsymbol{\lambda}) \big)- & \partial_2\mathbf{z}^{\boldsymbol\epsilon }_{t} \big( \lambda_{h,i},\lambda_{h',j}\big) \Big|ds \ |{X}_{h',j,t}^{N,\boldsymbol\epsilon,1}(\boldsymbol{\lambda})| \\
& \leq \frac{C}{N} \sum_{h',j}  \Big(| {X}_{h',j,t}^{N,\boldsymbol\epsilon,1}|^2 +\frac{1}{N} |{X}_{h',j,t}^{N,\boldsymbol\epsilon,2}| | {X}_{h',j,t}^{N,\boldsymbol\epsilon,1}| \Big)
\end{split}
\end{equation*}

\noindent to obtain
\begin{equation}\label{mstreq}
\begin{split}
\frac{d}{dt} \norm{{X}_{t}^{N,\boldsymbol\epsilon,2}}_{L^2( \mathbb{P}_{V,B}^{N,\boldsymbol\epsilon})}^2 &= 2 \ \mathbb{E}_{V,B}^{N,\boldsymbol\epsilon} \Bigg( \sum_{h,i} \dot{X}_{h,i,t}^{N,\boldsymbol\epsilon,2} \ {X}_{h,i,t}^{N,\boldsymbol\epsilon,2} \Bigg) \\
& \leq C \ \mathbb{E}_{V,B}^{N,\boldsymbol\epsilon} \Bigg( \sum_{h,i} |{X}_{h,i,t}^{N,\boldsymbol\epsilon,1}| \ |{X}_{h,i,t}^{N,\boldsymbol\epsilon,2}| \Bigg) +  \frac{C}{N} \  \mathbb{E}_{V,B}^{N,\boldsymbol\epsilon} \Bigg( \sum_{h,i} |{X}_{h,i,t}^{N,\boldsymbol\epsilon,2}|^2 \Bigg) \\
& +  C \ \mathbb{E}_{V,B}^{N,\boldsymbol\epsilon} \Bigg( \sum_{h,i}|{X}_{h,i,t}^{N,\boldsymbol\epsilon,1}|^2 \ |{X}_{h,i,t}^{N,\boldsymbol\epsilon,2}|  \Bigg) + \frac{C}{N} \ \mathbb{E}_{V,B}^{N,\boldsymbol\epsilon} \Bigg( \sum_{h,i}|{X}_{h,i,t}^{N,\boldsymbol\epsilon,1}| \ |{X}_{h,i,t}^{N,\boldsymbol\epsilon,2}|^2  \Bigg) \\
&+ \frac{C}{N^2} \ \mathbb{E}_{V,B}^{N,\boldsymbol\epsilon} \Bigg( \sum_{h,i}|{X}_{h,i,t}^{N,\boldsymbol\epsilon,2}|^3  \Bigg) \\
&+ \frac{C}{N} \ \mathbb{E}_{V,B}^{N,\boldsymbol\epsilon} \Bigg( \sum_{\substack{h,i \\ h',j }}|{X}_{h,i,t}^{N,\boldsymbol\epsilon,1}| \ |{X}_{h',j,t}^{N,\boldsymbol\epsilon,1}| \ |{X}_{h,i,t}^{N,\boldsymbol\epsilon,2}|  \Bigg) + \frac{C}{N^2} \ \mathbb{E}_{V,B}^{N,\boldsymbol\epsilon} \Bigg( \sum_{\substack{h,i \\ h',j }}|{X}_{h,i,t}^{N,\boldsymbol\epsilon,1}| \ |{X}_{h',j,t}^{N,\boldsymbol\epsilon,2}| \ |{X}_{h,i,t}^{N,\boldsymbol\epsilon,2}|  \Bigg)\\
&+ \frac{C}{N^2} \ \mathbb{E}_{V,B}^{N,\boldsymbol\epsilon} \Bigg( \sum_{\substack{h,i \\ h',j }}|{X}_{h,i,t}^{N,\boldsymbol\epsilon,2}|^2 \ |{X}_{h',j,t}^{N,\boldsymbol\epsilon,1}|  \Bigg) + \frac{C}{N^3} \ \mathbb{E}_{V,B}^{N,\boldsymbol\epsilon} \Bigg( \sum_{\substack{h,i \\ h',j }}|{X}_{h,i,t}^{N,\boldsymbol\epsilon,2}|^2 \ |{X}_{h',j,t}^{N,\boldsymbol\epsilon,2}|  \Bigg)\\
&+ \mathbb{E}_{V,B}^{N,\boldsymbol\epsilon} \Bigg( \sum_{h,i} \Big| \int \partial_1\mathbf{z}^{\boldsymbol\epsilon }_{t} \big(\lambda_{h,i} ,y \big) dM_N(y)  \Big| |{X}_{h,i,t}^{N,\boldsymbol\epsilon,1}| \ |{X}_{h,i,t}^{N,\boldsymbol\epsilon,2}| \Bigg) + C \  \mathbb{E}_{V,B}^{N,\boldsymbol\epsilon} \Bigg( \sum_{h,i} |{X}_{h,i,t}^{N,\boldsymbol\epsilon,2}|^2 \Bigg) \\
&+ \frac{C}{N} \ \mathbb{E}_{V,B}^{N,\boldsymbol\epsilon} \Bigg( \sum_{\substack{h,i \\ h',j }}|{X}_{h,i,t}^{N,\boldsymbol\epsilon,2}| \ |{X}_{h',j,t}^{N,\boldsymbol\epsilon,1}|^2  \Bigg) + \frac{C}{N^2} \ \mathbb{E}_{V,B}^{N,\boldsymbol\epsilon} \Bigg( \sum_{\substack{h,i \\ h',j }}|{X}_{h,i,t}^{N,\boldsymbol\epsilon,2}| \ |{X}_{h',j,t}^{N,\boldsymbol\epsilon,2}|\ |{X}_{h',j,t}^{N,\boldsymbol\epsilon,1}|  \Bigg)\\
&+ \frac{C}{N} \ \mathbb{E}_{V,B}^{N,\boldsymbol\epsilon} \Bigg( \sum_{\substack{h,i \\ h',j }}|{X}_{h,i,t}^{N,\boldsymbol\epsilon,2}| \ |{X}_{h',j,t}^{N,\boldsymbol\epsilon,2}|  \Bigg) .
\end{split}
\end{equation}

\noindent Using the bounds   $\norm{ \int \partial_1\mathbf{z}^{\boldsymbol\epsilon }_{t} \big(\lambda_{h,i} ,y \big) dM_N (y) }_{L^4( \mathbb{P}_{V,B}^{N,\boldsymbol\epsilon})} \leq C \log N$ (see Lemma \ref{controlefourier}),  $|{X}_{h,i,t}^{N,\boldsymbol\epsilon,1}| \leq C \ N$ , $|{X}_{h,i,t}^{N,\boldsymbol\epsilon,2}| \leq  C N^2$  and inequalities such as

\begin{equation*}
\begin{split}
\sum_{h,i} |{X}_{h,i,t}^{N,\boldsymbol\epsilon,1}| \ |{X}_{h,i,t}^{N,\boldsymbol\epsilon,2}| & \leq \frac{1}{2} \Big( \sum_{h,i} \big( ({X}_{h,i,t}^{N,\boldsymbol\epsilon,1})^2 +  ({X}_{h,i,t}^{N,\boldsymbol\epsilon,2})^2 \big) \Big) \\
\sum_{\substack{h,i \\ h',j }}|{X}_{h,i,t}^{N,\boldsymbol\epsilon,1}| \ |{X}_{h',j,t}^{N,\boldsymbol\epsilon,1}| \ |{X}_{h,i,t}^{N,\boldsymbol\epsilon,2}| & \leq \Big( \sum_{\substack{h,i \\ h',j }}\big(({X}_{h,i,t}^{N,\boldsymbol\epsilon,1})^4 + ({X}_{h,i,t}^{N,\boldsymbol\epsilon,1})^4 + ({X}_{h',j,t}^{N,\boldsymbol\epsilon,2})^2 \big) \Big)
\end{split}
\end{equation*}

\noindent along with (\ref{controle11}) and Hölder inequality, we get
\begin{equation}
\frac{d}{dt} \norm{{X}_{t}^{N,\boldsymbol\epsilon,2}}_{L^2( \mathbb{P}_{V,B}^{N,\boldsymbol\epsilon})}^2 \leq C \Big( \norm{{X}_{t}^{N,\boldsymbol\epsilon,2}}_{L^2( \mathbb{P}_{V,B}^{N,\boldsymbol\epsilon})}^2 + N (\log N)^4 \Big) .
\end{equation}

\noindent Using Grönwall's Lemma, we can conclude the proof. The bounds (\ref{controleespace2}) and (\ref{controlenorme}) are proven the same way.
\end{proof}

\begin{rem}\label{normelp}
Using  (\ref{controlenorme}), (\ref{equation1}) and (\ref{mstreq}) we see that we have in fact for all integer $k \geq 1$

\begin{equation*}
\sup_{\substack{0\leq h \leq g \\ 1\leq i \leq N_{h}}} \norm{{X}_{h,i,t}^{N,\boldsymbol\epsilon,1}}_{L^{2k}( \mathbb{P}_{V,B}^{N,\boldsymbol\epsilon})} \leq C_k \log N \ \ \ , \ \ \ \sup_{\substack{0\leq h \leq g \\ 1\leq i \leq N_{h}}}\norm{{X}_{h,i,t}^{N,\boldsymbol\epsilon,2}}_{L^{2k}( \mathbb{P}_{V,B}^{N,\boldsymbol\epsilon})} \leq C_k \sqrt{N}(\log N)^2 
\end{equation*}
\end{rem}

\section{From Transport to Universality}

\noindent In this section we will prove Proposition \ref{theoremfixed} and Corollary \ref{corfixed}. We prove  the results in the bulk as the proof is almost  identical for the edge  result.

\begin{proof}[Proof of Proposition \ref{theoremfixed}]

\noindent  Note that by Lemma \ref{lemmatransp} and by our construction of $\mathbf{Y}_t^{N,\boldsymbol\epsilon}$, ${X}_{1}^{N,\boldsymbol\epsilon}$ is an approximate transport map from  $\mathbb{P}_{V,B}^{N,\boldsymbol\epsilon}$ to  $\mathbb{P}_{T_1,B}^{N,\boldsymbol\epsilon}$ in the sense that it satisfies (\ref{almosttransp}). Now, keeping our notations from the previous section, set $\hat{X}^{N,\boldsymbol\epsilon} = Id + \frac{1}{N}{X}_1^{N,\boldsymbol\epsilon,1}$. Then for all $f \in C^1(\mathbb{R})$

\begin{equation*}
\begin{split}
\Big| \int f(\hat{X}^{N,\boldsymbol\epsilon}) d\mathbb{P}_{V,B}^{N,\boldsymbol\epsilon} - \int f({X}_1^{N,\boldsymbol\epsilon}) d\mathbb{P}_{V,B}^{N,\boldsymbol\epsilon} \Big| & \leq \frac{\norm{\nabla{f}}_{\infty}}{N^2} \int |{X}_1^{N,\boldsymbol\epsilon,2}| d\mathbb{P}_{V,B}^{N,\boldsymbol\epsilon} \\
& \leq \frac{\norm{\nabla{f}}_{\infty}}{N^2} \norm{{X}_{1}^{N,\boldsymbol\epsilon,2}}_{L^2( \mathbb{P}_{V,B}^{N,\boldsymbol\epsilon})} \\
 & \leq \norm{\nabla{f}}_{\infty} \frac{(\log N)^2}{N^{\frac{3}{2}}}
\end{split}
\end{equation*}

\noindent and thus 

\begin{equation*}
\Big| \int f(\hat{X}^{N,\boldsymbol\epsilon}) d\mathbb{P}_{V,B}^{N,\boldsymbol\epsilon} - \int f d\mathbb{P}_{T_1,B}^{N,\boldsymbol\epsilon} \Big|  \leq C \frac{(\log N)^3}{N}\norm{f}_{\infty} +  \norm{\nabla{f}}_{\infty} \frac{(\log N)^2}{N^{\frac{3}{2}}} .
\end{equation*}

\noindent Now for all $ 0 \leq h \leq g$ let $R^h: B_h^{N_h} \longrightarrow B_h^{N_h}$ the ordering map (i.e the map satisfying for all $(\lambda_{1},\cdots ,\lambda_{N_h}) \in B_h^{N_h}$ $ R^{h,i}(\lambda_{1},\cdots ,\lambda_{N_h}) \leq  R^{h,j}(\lambda_{1},\cdots ,\lambda_{N_h})$ if  $i<j$ and $\{\lambda_{1}, \cdots , \lambda_{Nh} \} = \{R^{h,1}(\lambda_{1},\cdots ,\lambda_{N_h}), \cdots , R^{h,N_h}(\lambda_{1},\cdots ,\lambda_{N_h}) \} $, so that if $R(\boldsymbol{\lambda})=  (R^0(\lambda_{0,1},\cdots ,\lambda_{0,N_0}),\cdots ,R^g(\lambda_{g,1},\cdots ,\lambda_{g,N_g})) $ we have $R \sharp d\mathbb{P}_{B}^{N,\boldsymbol\epsilon} = d\tilde{\mathbb{P}}_{B}^{N,\boldsymbol\epsilon}$ .\\

\noindent  Then  if $f_h $ is a function of $m$ variables, we have $\norm{\nabla(f_h \circ R^h)}_{\infty} \leq \sqrt{m} \norm{\nabla f_h}_{\infty}$. \\

\noindent It is clear from (\ref{controleespace}) that $\hat{X}^{N,\boldsymbol\epsilon}$ preserves the order of the eigenvalues with probability greater than $1 - N^{-\frac{N}{C}}$. Thus, if we define $f: \mathbb{R}^{m (g+1)}\longrightarrow \mathbb{R}$ by $f(\boldsymbol{x}_{0},\cdots,\boldsymbol{x}_{g}) = \prod_{0 \leq h \leq g} f_h(\boldsymbol{x}_h)$ where $f_h: \mathbb{R}^{m }\longrightarrow \mathbb{R}$ we obtain

\begin{equation}
\begin{split}
\Bigg| &  \int \prod_{0 \leq h \leq g} f_h\big(N(\lambda_{h,i_h+1}- \lambda_{h,i_h}),  \cdots,N(\lambda_{h,i_h+m} - \lambda_{h,i_h})\big)d\tilde{\mathbb{P}}_{T_1,B}^{N,\boldsymbol\epsilon} \\
   - &  \int \prod_{0 \leq h \leq g} f_h\big(N (\hat{X}_{h,i_h+1}^{N,\boldsymbol\epsilon}(\boldsymbol{\lambda})-\hat{X}_{h,i_h}^{N,\boldsymbol\epsilon}(\boldsymbol{\lambda})),\cdots,N (\hat{X}_{h,i_h+m}^{N,\boldsymbol\epsilon}(\boldsymbol{\lambda}) - \hat{X}_{h,i_h}^{N,\boldsymbol\epsilon}(\boldsymbol{\lambda}))\big)d\tilde{\mathbb{P}}_{V,B}^{N,\boldsymbol\epsilon} \Bigg| \\
  &  \leq C \Big( \frac{(\log N)^3}{N}\norm{f}_{\infty} +  \norm{\nabla{f}}_{\infty} \sqrt{m} \  \frac{(\log N)^2}{N^{\frac{1}{2}}} \Big) .
\end{split}
\end{equation}

\noindent Now, using (\ref{controleespace}) we notice that with probability greater than $1 - N^{-\frac{N}{C}}$, for all $ 1 \leq k \leq m$ and $ 0 \leq h \leq g$ 

\begin{equation*}
\hat{X}_{h,i_h+k}^{N,\boldsymbol\epsilon}(\boldsymbol{\lambda})-\hat{X}_{h,i_h}^{N,\boldsymbol\epsilon}(\boldsymbol{\lambda}) = \lambda_{h,i+k} - \lambda_{h,i_h} + (\lambda_{h,i_h+k} - \lambda_{h,i_h}) O(\frac{\log N}{\sqrt{N}}) .
\end{equation*}
 
 \noindent As  $f_h$ has compact support in $[-M,M]^m$,  $(\lambda_{h,i_h+k} - \lambda_{h,i_h})$ remains bounded by $\frac{2 M}{N}$ and 
 
 \begin{equation*}
\hat{X}_{h,i+k}^{N,\boldsymbol\epsilon}(\boldsymbol{\lambda})-\hat{X}_{h,i_h}^{N,\boldsymbol\epsilon}(\boldsymbol{\lambda}) = \lambda_{h,i+k} - \lambda_{h,i_h} +  O(\frac{M  \log N}{N \sqrt{N}}) ,
\end{equation*}

\noindent we easily deduce the first part of Proposition \ref{theoremfixed} .\\

\end{proof}

\noindent Before proving Corollary \ref{corfixed} we recall Theorem 1.5 of \cite{BFG}.

\begin{prop}\label{thmonecut}
 Assume that $W$ is a potential satisfying Hypothesis \ref{hypo} with $g=0$. Then for a constant $C$ and for all $m \in \mathbb{N}^*$ and $f: \mathbb{R}^m \longrightarrow \mathbb{R}$ Lipschitz and compactly supported in $[-M ; M]$ we have  \\
\begin{enumerate}
\item In the Bulk 
\begin{equation*}
\begin{split}
\Bigg|  \int & f\big(N(\lambda_{i+1}- \lambda_{i}),  \cdots,N(\lambda_{i+m} - \lambda_{i})\big)d\tilde{\mathbb{P}}_{W}^{N} \\
   -  &   \int  f\big(N (\Phi)'(\lambda_{i})(\lambda_{i+1}-\lambda_{i}),\cdots,N (\Phi)'(\lambda_{i})(\lambda_{i+m} - \lambda_{i})\big)d\tilde{\mathbb{P}}_{G}^{N}  \Bigg| \\
  \leq & C \frac{(\log N)^3}{N}\norm{f}_\infty + C(\sqrt{m}\frac{(\log N)^2}{N^{1/2}} + M \frac{(\log N)}{N^{1/2}} +\frac{M^2}{N} ) \norm{\nabla f}_\infty
\end{split}
\end{equation*}
\item At the Edge
\begin{equation*}
\begin{split}
\Bigg|  \int &   f\big(N^{2/3}(\lambda_{1}-  \alpha_{-}),  \cdots,N^{2/3}(\lambda_{m} - \alpha_{-})\big)d\tilde{\mathbb{P}}_{W}^N \\
   -  &   \int  f\big(N^{2/3} (\Phi)'(-2)(\lambda_{1}+2),\cdots,N^{2/3} (\Phi)'(-2)(\lambda_{m} + 2)\big)d\tilde{\mathbb{P}}_{G}^{N}  \Bigg| \\
  \leq & C \frac{(\log N)^3}{N}\norm{f}_\infty + C(\sqrt{m}\frac{(\log N)^2}{N^{5/6}}  +\frac{\log N}{N^{1/3}} +  \frac{M^2}{N^{4/3}}) \norm{\nabla f}_\infty
\end{split}
\end{equation*}

 \end{enumerate}
 \noindent  where $\Phi$  is a transport map from $\mu_G$   to $\mu_{W}$, and we recall that $G$ denotes the Gaussian potential.
\end{prop}

\begin{proof}[Proof of Corollary \ref{corfixed}]

Noticing that $d{\mathbb{P}}_{T_1,B}^{N,\boldsymbol\epsilon}$ is a product measure we can write

\begin{equation*}
\begin{split}
\int & \prod_{0 \leq h \leq g} f_h\big(N (\lambda_{h,i_h+1}-\lambda_{h,i_h}),\cdots,N (\lambda_{h,i_h+m} - \lambda_{h,i_h})\big)d{\mathbb{P}}_{T_1,B}^{N,\boldsymbol\epsilon} 
=   \frac{1}{Z_{T_1,B}^{N,\boldsymbol\epsilon}} \int \prod_{0 \leq h \leq g} \prod_{1 \leq i \leq N_h} \mathbf{1}_{B_h}(\lambda_{h,i}) d\lambda_{h,i} \\
& \bigg[ f_h\big(N (\lambda_{h,i_h+1}-\lambda_{h,i_h}),\cdots,N (\lambda_{h,i_h+m} - \lambda_{h,i_h})\big) \prod_{1\leq i < j \leq N_h} \lvert \lambda_{h,i} - \lambda_{h,j} \rvert^\beta  \exp \Big( -  N \sum_{1\leq i \leq N_h} \tilde{V}^{\boldsymbol\epsilon}(\lambda_{h,i}) \Big) \bigg] \\ 
= & \prod_{0 \leq h \leq g} \int   f_h\big({N} (\lambda_{h,i_h+1}-\lambda_{h,i_h}),\cdots,N (\lambda_{h,i_h+m} - \lambda_{h,i_h})\big)d{\mathbb{P}}_{{\tilde{V}^{\boldsymbol\epsilon}}/{\epsilon_h},B_h}^{N_h} 
\end{split}
\end{equation*}

\noindent We notice using (\ref{caracterisation}) that 
\begin{equation*}
\mu_{{\tilde{V}^{\boldsymbol\epsilon}}/{\epsilon_h},B_h}= \mu_{V}^{\boldsymbol\epsilon,h}.
\end{equation*}

\noindent We conclude using  Proposition \ref{thmonecut}.
\end{proof}

\section{Universality in the initial model}
\noindent To derive universality in the initial model, we expand the expectation of the quantity we want to compute in terms of  the filling fractions, and we  make use of  Corollary \ref{corfixed}.\\

\noindent First, we notice  that  for all $0\leq h \leq g$ the map $\Phi^{\boldsymbol\epsilon,h }$ is smooth in $\boldsymbol\epsilon \in \tilde{\mathcal{E}}$ and we have a bound 
\begin{equation}\label{etoile}
(\Phi^{\boldsymbol\epsilon,h })'(\lambda_{h,i}) = (\Phi^{\boldsymbol\epsilon_\star,h })'(\lambda_{h,i}) + O(|\boldsymbol\epsilon - \boldsymbol\epsilon_\star |)  \ \ \ \ \ \ \ \ \ uniformely \ in \ \lambda_{h,i}\in B
\end{equation}

\noindent Indeed, it is shown in \cite{BFG} (4.1) that our transport map $\Phi^{\boldsymbol\epsilon,h }$ is equal to $X_1^{\boldsymbol\epsilon}$ where $X_t^{\boldsymbol\epsilon}$ solves the ordinary diffential equation 

\begin{equation*}
\dot{X}_t^{\boldsymbol\epsilon} = {y}_t^{\boldsymbol\epsilon} ({X}_t^{\boldsymbol\epsilon}) \ \ , \ \ {X}_0^{\boldsymbol\epsilon} = Id
\end{equation*}

\noindent and  ${y}_t^{\boldsymbol\epsilon}$ is given by inverting $\Xi$. By formula (\ref{formuleinverse}) and Lemma \ref{regularity}, we see that ${y}_t^{\boldsymbol\epsilon}$ is regular in $\epsilon$, and from the standard theory of ordinary differential equations, so is $\Phi^{\boldsymbol\epsilon}$. \\

\noindent We will   use  the following result proved in section 8.2, equations (8.18) and (8.19) of \cite{BGII}. 

\begin{lem}\label{GaussDiscret}
Along the subsequences such that $\boldsymbol{N_\star}\mod \mathbb{Z}^{g+1} \longrightarrow \kappa $ where $\kappa \in [0;1[^{g+1}$ and under ${\mathbb{P}}_{V,\mathcal{B}}^{N}$,  the vector $  \lfloor \boldsymbol{N_\star} \rfloor - N(\boldsymbol{\lambda}) $ converges towards a random discrete Gaussian vector $\Delta_{h, \kappa}$. In particular 
\begin{equation*}
\mathbb{P}_{V,{B}}^N \big( | \mathrm{N}(\boldsymbol\lambda) - \lfloor \boldsymbol{N_\star} \rfloor |  \geq K )= O \Big( \exp ( - K^2) \Big) .
\end{equation*}
\end{lem}

\noindent Note that the limit is not necessarily centered, and although the result is proved for $\boldsymbol{N_\star}  - N(\boldsymbol{\lambda})$, it obviously also holds for  $\lfloor \boldsymbol{N_\star} \rfloor - N(\boldsymbol{\lambda})$ since we are only considering subsequences such that $ \boldsymbol{N_\star}  - \lfloor \boldsymbol{N_\star} \rfloor  \longrightarrow \kappa$.
\noindent We will also need the following result, which can be proved using the previous result or  Lemma  \ref{controlefourier}

\begin{equation}\label{ecarteps}
 \sum_{\substack{\mathbf{N}=(N_0,\cdots,N_g) }} \frac{N!}{\prod N_h!} \ \  \frac{Z_{V,B}^{N,\boldsymbol\epsilon}}{Z_{V,B}^{N}} |\boldsymbol\epsilon - \boldsymbol\epsilon_\star | = \mathbb{E}_{V,B}^{N} \Big( \sum_{0 \leq h \leq g} | L_N(B_h) - \mu_{V}(B_h) |\Big) \leq C \frac{\log N}{N} .
\end{equation}

\noindent We now provide a proof of  Theorem \ref{theoreme}.  Let $f$ be a function of compact support and $i$ such as in the hypothesis of the theorem. Using Corollary \ref{corfixed} we have

\begin{equation*}
\begin{split}
& \ \ \ \int    f\big(N \rho_V(E^{V,N}_{i}) (\lambda_{i+1}-\lambda_{i})\big)d\tilde{\mathbb{P}}_{V,B}^{N} \\
= & \sum_{\mathbf{N}=(N_0,\cdots,N_g)} \int   f\big(N \rho_V(E^{V,N}_{i})(\lambda_{i+1}-\lambda_{i})\big) \mathbf{1}_{\mathrm{N}(\boldsymbol\lambda) = \mathbf{N}} \  d\tilde{\mathbb{P}}_{V,B}^{N} \\
= & \sum_{\mathbf{N}=(N_0,\cdots,N_g)} \frac{N!}{\prod N_h!} \ \  \frac{Z_{V,B}^{N,\boldsymbol\epsilon}}{Z_{V,B}^{N}} \int   f \big(N \rho_V(E^{V,N}_{i}) (\lambda_{h,i+1[h,\mathbf{N}]}-\lambda_{h,i[h,\mathbf{N}]})\big)  d\tilde{\mathbb{P}}_{V,B}^{N,\boldsymbol\epsilon} \\
= & \sum_{\substack{\mathbf{N}=(N_0,\cdots,N_g) \\ | \mathrm{N}(\boldsymbol\lambda) - \lfloor \boldsymbol{N_\star} \rfloor | \leq K }} \frac{N!}{\prod N_h!} \ \  \frac{Z_{V,B}^{N,\boldsymbol\epsilon}}{Z_{V,B}^{N}} \int   f\big(N \rho_V(E^{V,N}_{i}) (\lambda_{h,i+1[h,\mathbf{N}]}-\lambda_{h,i[h,\mathbf{N}]})\big) d\tilde{\mathbb{P}}_{V,B}^{N,\boldsymbol\epsilon} \\[0.5em]
&  + O\Big(\norm{f}_{\infty} \exp ( - K^2)\Big) \\
= &  \sum_{\substack{\mathbf{N}=(N_0,\cdots,N_g) \\ | \mathrm{N}(\boldsymbol\lambda) - \lfloor \boldsymbol{N_\star} \rfloor | \leq K }} \frac{N!}{\prod N_h!} \ \  \frac{Z_{V,B}^{N,\boldsymbol\epsilon}}{Z_{V,B}^{N}}  \int   f \big(N (\Phi^{\boldsymbol\epsilon,h })' (\lambda_{i[h,\mathbf{N}]})  \rho_V(E^{V,N}_{i}) (\lambda_{i+1[h,\mathbf{N}]}-\lambda_{i[h,\mathbf{N}]})\big) d\tilde{\mathbb{P}}_{G}^{N_h}  \\[0.5em]
& +O\Big( (\exp ( - K^2) + \frac{(\log N)^3}{N}) \norm{f}_\infty + (\sqrt{m}\frac{(\log N)^2}{N^{1/2}} + M \frac{(\log N)}{N^{1/2}} +\frac{M^2}{N}) \norm{\nabla f}_\infty \Big) .
\end{split}
\end{equation*}

\noindent If we manage to replace the term $N (\Phi^{\boldsymbol\epsilon,h })' (\lambda_{i[h,\mathbf{N}]})  \rho_V(E^{V,N}_{i})$ by  $N_h \ \rho_G(E^{G,N_h}_{i[h,\mathbf{N}]})$ then, using the convergence (\ref{convergencegap}) we can conclude.\\

\noindent By (\ref{etoile}) we can replace $(\Phi^{\boldsymbol\epsilon,h })' (\lambda_{i[h,\mathbf{N}]}) $ by $(\Phi^{\boldsymbol\epsilon_\star,h })' (\lambda_{i[h,\mathbf{N}]}) $ in the last equation and obtain an error of order $K/N$. Now, using that 
 $\Phi^{\boldsymbol\epsilon_\star,h }$ is  a transport from $\mu_G$ to $\mu_V^{\boldsymbol\epsilon_\star,h }$ we see that
 
 \begin{equation*}
 \begin{split}
 (\Phi^{\boldsymbol\epsilon_\star,h })' (\lambda_{i[h,\mathbf{N}]}) &= \frac{\rho_G(\lambda_{i[h,\mathbf{N}]})}{\rho_V^{\boldsymbol\epsilon_\star,h} (\Phi^{\boldsymbol\epsilon_\star,h }(\lambda_{i[h,\mathbf{N}]}))}  \ , \\
  \int_{-\infty}^{\Phi^{\boldsymbol\epsilon_\star,h }(E^{G,N_h}_{i[h,\mathbf{N}]})} \rho_V^{\boldsymbol\epsilon_\star,h} (x) dx &= \int_{-\infty}^{E^{V,N}_i} \rho_V^{\boldsymbol\epsilon_\star,h} (x) dx + O(K/N) .
 \end{split}
\end{equation*}  

\noindent Thus $\Phi^{\boldsymbol\epsilon_\star,h }(E^{G,N_h}_{i[h,\mathbf{N}]}) =E^{V,N}_i + O(K/N)$ and using  $\rho_V = \boldsymbol\epsilon_{\star,h}\  \rho_V^{\boldsymbol\epsilon_\star,h} $ on $A_h$ we see that
 \begin{equation*}
N (\Phi^{\boldsymbol\epsilon_\star,h })' (\lambda_{i[h,\mathbf{N}]})  \rho_V(E^{V,N}_{i}) = N_{\star,h} \ \rho_G(\lambda_{i[h,\mathbf{N}]}) \ \frac{\rho_V^{\boldsymbol\epsilon_\star,h}(E^{V,N}_{i}) }{\rho_V^{\boldsymbol\epsilon_\star,h} (\Phi^{\boldsymbol\epsilon_\star,h }(\lambda_{i[h,\mathbf{N}]}))}  .
\end{equation*}  

\noindent We can replace $\lambda_{i[h,\mathbf{N}]}$ by  $E^{G,N_h}_{i[h,\mathbf{N}]}$ in the right hand side with an  error term  $o(N)$ with high probability under ${\mathbb{P}}_{G}^{N_h}$ using a very rough rigidity  estimate  that can be proved for instance using Proposition \ref{concentration}. As   $(\Phi^{\boldsymbol\epsilon_\star,h })'$ is bounded by below  and  $f$ is compact we notice that $N (\lambda_{i+1[h,\mathbf{N}]}-\lambda_{i[h,\mathbf{N}]})$ is of order $1$ and we can conclude.\\

\noindent  We can now proceed with the proof of Theorem \ref{theoremedge}. To simplify the notations, we will do the proof when $m=1$ but the proof for general $m$ is identical.

\begin{equation*}
\begin{split}
& \ \ \ \int  \prod_{h=0}^g f_h\big(N^{2/3} (\lambda_{h,1}-\alpha_{0,-})\big)d\tilde{\mathbb{P}}_{V,B}^{N} \\
= & \sum_{\mathbf{N}=(N_0,\cdots,N_g)} \int  \prod_{h=0}^g  f_h\big(N^{2/3} (\lambda_{h,1}-\alpha_{0,-})\big) \mathbf{1}_{\mathrm{N}(\boldsymbol\lambda) = \mathbf{N}} \  d\tilde{\mathbb{P}}_{V,B}^{N} \\
= & \sum_{\mathbf{N}=(N_0,\cdots,N_g)} \frac{N!}{\prod N_h!} \ \  \frac{Z_{V,B}^{N,\boldsymbol\epsilon}}{Z_{V,B}^{N}} \int \prod_{h=0}^g  f_h\big(N^{2/3} (\lambda_{h,1}-\alpha_{0,-})\big)  d\tilde{\mathbb{P}}_{V,B}^{N,\boldsymbol\epsilon} \\
= & \sum_{\mathbf{N}=(N_0,\cdots,N_g)} \frac{N!}{\prod N_h!} \ \  \frac{Z_{V,B}^{N,\boldsymbol\epsilon}}{Z_{V,B}^{N}} \int \prod_{h=0}^g  f_h\big(N^{2/3} (\lambda_{h,1}-\alpha^{\boldsymbol\epsilon}_{0,-})\big)  d\tilde{\mathbb{P}}_{V,B}^{N,\boldsymbol\epsilon} \\
&  + O\Big(\sum_{\mathbf{N}=(N_0,\cdots,N_g)} \frac{N!}{\prod N_h!} \ \  \frac{Z_{V,B}^{N,\boldsymbol\epsilon}}{Z_{V,B}^{N}}    \norm{\nabla f}_\infty N^{2/3} |\boldsymbol\epsilon - \boldsymbol\epsilon_\star | \Big) \\
= & \sum_{\substack{\mathbf{N}=(N_0,\cdots,N_g) \\ | \mathrm{N}(\boldsymbol\lambda) - \lfloor \boldsymbol{N_\star} \rfloor | \leq K}} \frac{N!}{\prod N_h!} \ \  \frac{Z_{V,B}^{N,\boldsymbol\epsilon}}{Z_{V,B}^{N}} \int  \prod_{h=0}^g  f_h\big(N^{2/3} (\lambda_{h,1}-\alpha^{\boldsymbol\epsilon}_{h,-})\big) d\tilde{\mathbb{P}}_{V,B}^{N,\boldsymbol\epsilon} \\[0.5em]
&  + O\Big(\frac{\log N}{N^{1/3}}  \norm{\nabla f}_\infty  + \norm{f}_{\infty}  \exp( - K^2)\Big) \\
= &  \sum_{\substack{\mathbf{N}=(N_0,\cdots,N_g) \\ | \mathrm{N}(\boldsymbol\lambda) - \lfloor \boldsymbol{N_\star} \rfloor | \leq K}} \frac{N!}{\prod N_h!} \ \  \frac{Z_{V,B}^{N,\boldsymbol\epsilon}}{Z_{V,B}^{N}} \prod_{h=0}^g  \int   f_h\big(N^{2/3} (\Phi^{\boldsymbol\epsilon,h })'(-2) (\lambda_{1} + 2 )\big) d\tilde{\mathbb{P}}_{G}^{N_h}  \\[0.5em]
& +O\Big( (\exp( - K^2) +  \frac{(\log N)^3}{N}) \norm{f}_\infty + (\sqrt{m}\frac{(\log N)^2}{N^{5/6}}  +\frac{\log N}{N^{1/3}} +  \frac{M^2}{N^{4/3}}) \norm{\nabla f}_\infty \Big) .
\end{split}
\end{equation*}

\noindent Using the fact that $(\Phi^{\boldsymbol\epsilon,h })'$ is bounded by below on $B$ and that  $f_h$ is supported in $[-M ; M]$ we obtain that $|\lambda_{1} + 2 |$ remains bounded by $\frac{C M}{N^{2/3}}$. Using (\ref{etoile}) we get 

\begin{equation*}
\begin{split}
   f_h \big(N^{2/3} (\Phi^{\boldsymbol\epsilon,h })'(-2) (\lambda_{1} + 2)\big) = &  f_h\big(N^{2/3} (\Phi^{\boldsymbol\epsilon_\star,0 })'(\alpha_{G,-}) (\lambda_{1}-\alpha_{G,-})\big)  \\
  + & O( M \norm{\nabla f}_\infty |\boldsymbol\epsilon - \boldsymbol\epsilon_\star |) .
 \end{split}
\end{equation*}

\noindent This equation, along with (\ref{ecarteps}),  shows that

\begin{equation*}
\begin{split}
& \ \ \ \int    \prod_{h=0}^g f_h\big(N^{2/3} (\lambda_{h,1}-\alpha_{h,-})\big)d\tilde{\mathbb{P}}_{V,B}^{N} \\
= &  \sum_{\substack{\mathbf{N}=(N_0,\cdots,N_g) \\ | \mathrm{N}(\boldsymbol\lambda) - \lfloor \boldsymbol{N_\star} \rfloor | \leq K}} \frac{N!}{\prod N_h!} \ \  \frac{Z_{V,B}^{N,\boldsymbol\epsilon}}{Z_{V,B}^{N}}  \prod_{h=0}^g  \int   f_h \big(N^{2/3} (\Phi^{\boldsymbol\epsilon,h })'(-2) (\lambda_{1} + 2 )\big) d\tilde{\mathbb{P}}_{G}^{N_h}  \\
  + & O\Big(  (\exp( - K^2) +  \frac{(\log N)^3}{N}) \norm{f}_\infty + (\sqrt{m}\frac{(\log N)^2}{N^{5/6}}  +\frac{\log N}{N^{1/3}} +  \frac{M^2}{N^{4/3}}) \norm{\nabla f}_\infty \Big) .
\end{split}
\end{equation*}

\noindent As Theorem 1.1 of  \cite{RRV} ensures the convergence  of the expectation,  we can conclude.\\

\noindent We now come to the proof of Theorem \ref{theoremedge2}. Let $0\leq h \leq g$, $i = [\lfloor \boldsymbol{N_\star} \rfloor]_{h-1} +1$ and $\Delta_{h}(\boldsymbol{\lambda}) = [\lfloor \boldsymbol{N_\star} \rfloor]_{h-1} - [N(\boldsymbol{\lambda})]_{h-1}$. As  before we obtain 

\begin{equation*}
\begin{split}
 \ \ \ \int  \ f\big(N^{2/3} (\lambda_{i}- \xi_h)\big)d\tilde{\mathbb{P}}_{V,B}^{N} &=  \sum_{\substack{\mathbf{N}=(N_0,\cdots,N_g) \\ \Delta_{h}(\boldsymbol{\lambda}) \geq 0}} \frac{N!}{\prod N_h!} \ \  \frac{Z_{V,B}^{N,\boldsymbol\epsilon}}{Z_{V,B}^{N}} \int  f\big(N^{2/3} (\lambda_{h, i[h,\mathbf{N}]} -\alpha_{h,-})\big)  d\tilde{\mathbb{P}}_{V,B}^{N,\boldsymbol\epsilon}  \\ 
 + &\sum_{\substack{\mathbf{N}=(N_0,\cdots,N_g) \\ \Delta_{h}(\boldsymbol{\lambda}) < 0}} \frac{N!}{\prod N_h!} \ \  \frac{Z_{V,B}^{N,\boldsymbol\epsilon}}{Z_{V,B}^{N}} \int   f\big(N^{2/3} (\lambda_{h - 1 , i[h-1 ,\mathbf{N}]} -\alpha_{h-1,+})\big)  d\tilde{\mathbb{P}}_{V,B}^{N,\boldsymbol\epsilon} .
\end{split}
\end{equation*}

\noindent We focus on the first term. Applying Corollary \ref{corfixed} we see that this term equals to
\begin{equation*}
\begin{split}
\sum_{\substack{| \mathrm{N}(\boldsymbol\lambda) - \lfloor \boldsymbol{N_\star} \rfloor | \leq K \\ \Delta_{h}(\boldsymbol{\lambda}) \geq 0}} \frac{N!}{\prod N_h!} \ \  \frac{Z_{V,B}^{N,\boldsymbol\epsilon}}{Z_{V,B}^{N}} \int  f\big(N^{2/3}  (\Phi^{\boldsymbol\epsilon,h })'(-2) (\lambda_{i[h,\mathbf{N}]} + 2) \big)  d\tilde{\mathbb{P}}_{G}^{N_h} . 
\end{split}
\end{equation*}

\noindent Noticing that $ i[h,\mathbf{N}] =  [\lfloor \boldsymbol{N_\star} \rfloor]_{h-1} - [\mathbf{N}]_{h-1} + 1$ , we deduce the  theorem from Lemma \ref{GaussDiscret}.

\bibliographystyle{abbrv}
\bibliography{biblio}

\end{document}